\documentclass{amsart}
\usepackage{amsmath}
\usepackage{amsthm}
\usepackage{tikz-cd}
\usepackage{tikz}
\usetikzlibrary{shapes.geometric}
\usepackage{eucal}
\usepackage{microtype}
\pdfoutput=1 
\usepackage{longtable,booktabs}

\usepackage{hyperref}
\usepackage{xcolor}
\hypersetup{colorlinks, linkcolor={red!50!black}, citecolor={blue!50!black}, urlcolor={blue!80!black}}

\usepackage[letterpaper, left = 1.5in, right = 1.5in, top = 1.5in, bottom = 1.5in]{geometry}

\usepackage[capitalise,nameinlink]{cleveref}
\crefname{subsection}{Subsection}{Subsections}
\crefname{subsubsection}{Subsubsection}{Subsubsections}
\crefname{table}{Table}{Tables}

\swapnumbers

\theoremstyle{definition}
\newtheorem{theorem}{Theorem}[subsection]
\newtheorem{defn}[theorem]{Definition}

\newtheorem{lemma}[theorem]{Lemma}
\newtheorem{prop}[theorem]{Proposition}
\newtheorem{rmk}[theorem]{Remark}
\newtheorem{construction}[theorem]{Construction}

\newcommand{\bbZ}{\mathbb{Z}}
\newcommand{\bbG}{\mathbb{G}}

\newcommand{\bbF}{\mathbb{F}}

\newcommand{\bbR}{\mathbb{R}}

\newcommand{\bbC}{\mathbb{C}}

\newcommand{\calA}{\mathcal{A}}

\newcommand{\calY}{\mathcal{Y}}

\newcommand{\Map}{\operatorname{Map}}

\newcommand{\Ext}{\operatorname{Ext}}
\newcommand{\Cof}{\operatorname{Cof}}
\newcommand{\coker}{\operatorname{Coker}}
\newcommand{\colim}{\operatorname*{colim}}

\newcommand{\im}{\operatorname{Im}}
\renewcommand{\ker}{\operatorname{Ker}}

\newcommand{\Sq}{\mathrm{Sq}}

\newcommand{\Spf}{\operatorname{Spf}}

\newcommand{\res}{\operatorname{res}}
\newcommand{\tr}{\operatorname{tr}}

\newcommand{\h}{\mathrm{h}}

\newcommand{\cl}{\mathrm{cl}}

\newcommand{\bs}{{-}}
 
\newcommand{\ul}{\underline}
\newcommand{\ol}[1]{[#1]}
\def\hyp{{\hbox{-}}}\newcommand{\per}{v}

\newcommand\xqed[1]{%
  \leavevmode\unskip\penalty9999 \hbox{}\nobreak\hfill
  \quad\hbox{#1}}
\newcommand\tqed{\xqed{$\triangleleft$}}

\makeatletter
\DeclareRobustCommand{\tvdots}{%
  \vbox{\baselineskip4\p@\lineskiplimit\z@\kern0\p@\hbox{.}\hbox{.}\hbox{.}}}
\makeatother

\makeatletter
\@namedef{subjclassname@2020}{\textup{2020} Mathematics Subject Classification}
\makeatother

\begin{document}

\title[The $C_2$-equivariant $K(1)$-local sphere]{The $C_2$-equivariant $K(1)$-local sphere}
\author[William Balderrama]{William Balderrama}

\begin{abstract}
We compute the $RO(C_2)$-graded Green functor $\ul{\pi}_\star L_{KU_{C_2}/(2)}S_{C_2}$.
\end{abstract}

\maketitle

\section{Introduction}

The $C_2$-equivariant stable stems are comprised of the abelian groups
\[
\pi_{c+w\sigma}S_{C_2} = \colim_{n\rightarrow\infty}\pi_0\Map_\ast^{C_2}(S^{n(1+\sigma)+(c+w\sigma)},S^{n(1+\sigma)})
\]
for $c,w\in\bbZ$, where $\sigma$ is the sign representation of $C_2$. These groups assemble to form the $RO(C_2)$-graded coefficient ring $\pi_\star S_{C_2}$ of the $C_2$-equivariant sphere spectrum.

The first systematic computations of $\pi_{\star}S_{C_2}$ were carried out by Araki and Iriye \cite{arakiiriye1982equivariant, iriye1982equivariant}, building on initial investigations of Bredon \cite{bredon1967equivariant} and Landweber \cite{landweber1969equivariant}. Recently there has been a renewed interest in computations of $\pi_{\star}S_{C_2}$, owing in part to its connection with Mahowald's root invariant \cite{mahowald1967metastable, mahowaldravenel1993root, brunergreenlees1995bredon} and $\bbR$-motivic homotopy theory \cite{duggerisaksen2017z2, behrensshah2020c2, belmontguillouisaksen2021c2}, and work of Belmont, Dugger, Guillou, Hill, Isaksen, and Ravenel \cite{duggerisaksen2017low, guillouisaksen2019bredonlandweber, guillouhillisaksenravenel2020cohomology, belmontisaksen2020rmotivic, belmontguillouisaksen2021c2} has begun to shed light on the computation of $\pi_\star S_{C_2}$ via $\bbR$-motivic and Adams spectral sequence methods.

The classical Adams spectral sequence and variants remain the most effective tool for $2$-primary stem-by-stem computations. However, our understanding of the broad-scale structure of the stable homotopy category is informed by the chromatic perspective, which filters stable homotopy theory by certain periodic phenomena organized by height. After rational stable homotopy theory at height $0$, this story begins in earnest at height $1$, which corresponds to the portion of stable homotopy theory detected by topological $K$-theory. In the stable stems, this is the image of $J$ and related classes.

We are interested in chromatic equivariant homotopy theory, where much less is known and far fewer computations have been carried out. An inspection of the $\bbR$-motivic Adams charts \cite{belmontisaksen2020rcharts} reveals that height $1$ classes constitute a large piece of $\pi_\star S_\bbR$, and the close relationship between $\bbR$-motivic and $C_2$-equivariant homotopy theory indicates that the same patterns should occur in $\pi_\star S_{C_2}$. After $2$-completion, the height $1$ portion of the sphere spectrum is captured by $S_{K(1)} = L_{KU/(2)}S$, the Bousfield localization of the sphere spectrum with respect to $K(1) = KU/(2)$ \cite{bousfield1979localization, ravenel1984localization}. The purpose of this paper is to compute the following $C_2$-equivariant analogue.

\begin{theorem}\label{thm:main}
The $RO(C_2)$-graded Green functor
\[
\ul{\pi}_\star L_{KU_{C_2}/(2)}S_{C_2}
\]
is as computed in \cref{sec:sphere} and summarized in \cref{ssec:tables}.
\tqed
\end{theorem}

\begin{rmk}\label{rmk:kr}
$KU_{C_2}$ refers to the $C_2$-equvariant spectrum representing the $K$-theory of $C_2$-equivariant complex vector bundles, not to Atiyah's $K\bbR$-theory. However, $KU_{C_2}/(2)$ is Bousfield equivalent to $K\bbR/(2)$: the Bousfield class of a $C_2$-spectrum is determined by the Bousfield classes of its underlying spectrum and geometric fixed points, and both $KU_{C_2}/(2)$ and $K\bbR/(2)$ have underlying spectrum $KU/(2)$ and contractible geometric fixed points. Thus our computation may also be interpreted as describing $\ul{\pi}_\star L_{K\bbR/(2)}S_{C_2}$.
\tqed
\end{rmk}

\begin{rmk}
Bhattacharya--Guillou--Li \cite{bhattacharyaguillouli2022rmotivic} have produced a $C_2$-equivariant complex $\calY = \smash{\calY^{C_2}_{(\mathsf{h},1)}}$ lifting the nonequivariant complex $Y = \bbR P^2 \wedge \bbC P^2$, together with a lift of the $v_1$-self map of $Y$ to a map $v\colon \Sigma^{1+\sigma}\calY \rightarrow \calY$ which is nilpotent on geometric fixed points. The positive resolution of the nonequivariant telescope conjecture at height $1$ implies, as in \cref{rmk:kr}, that $\calY[v^{-1}]$ is Bousfield equivalent to $KU_{C_2}/(2)$. Thus our computation may also be interpreted telescopically as describing $\ul{\pi}_\star L_{\calY[v^{-1}]}S_{C_2}$.
\tqed
\end{rmk}

The structure of $\ul{\pi}_\star L_{KU_{C_2}/(2)}S_{C_2}$ is rather more complex than the classical object $\pi_\ast S_{K(1)}$, and our methods are just robust enough that we are able to completely describe it, not merely up to extensions. This computation serves as a template for future work at different groups, primes, and heights. Our starting point is the following observation.

\begin{lemma}\label{lem:kuloc}
There is an equivalence
$
L_{KU_{C_2}/(2)}S_{C_2}\simeq F(EC_{2+},i_\ast S_{K(1)}).
$
\end{lemma}
\begin{proof}
As in \cref{rmk:kr}, $KU_{C_2}/(2)$ is Bousfield equivalent to $C_{2+}\otimes KU/(2)$. The lemma then follows from \cite[Proposition 3.22]{carrick2022smashing}.
\end{proof}

Abbreviate
\[
b(X) = F(EC_{2+},i_\ast X),
\]
so that $b(X)$ represents the Borel cohomology theory associated to $X$, and in particular
\[
\pi_{c+w\sigma}b(X) = [(S^{c+w\sigma})_{\h C_2},X] = [\Sigma^c P_w,X] = X^{-c}P_w,
\]
where $P_w = \bbR P^\infty_w$ is a stunted projective spectrum. The Segal conjecture for $C_2$, proved by Lin \cite{lin1980conjectures}, implies that $(S_{C_2})_2^\wedge\simeq b(S_2^\wedge)$. From this perspective, \cref{lem:kuloc} gives a quantitative description of the portion of $S_{C_2}$ detected by $L_{KU_{C_2}/(2)}S_{C_2}$. As a basic example, if $\alpha\in\pi_\star S_{C_2}$ lifts a class in $\pi_\ast S$ detected by $S_{K(1)}$, then $\alpha$ is itself detected by $L_{KU_{C_2}/(2)}S_{C_2}$. In general, we expect that knowledge of $\ul{\pi}_{\star}L_{KU_{C_2}/(2)}S_{C_2}$ will play a similar but more involved role in the study of $\ul{\pi}_\star S_{C_2}$ as knowledge of $\pi_\ast S_{K(1)}$ plays in the study of $\pi_\ast S$.

The groups $\pi_\star L_{KU_{C_2}/(2)} S_{C_2} \cong \pi_\star b(S_{K(1)}) \cong S_{K(1)}^\ast P_\ast$ may be computed by the familiar homotopy fixed point spectral sequences (HFPSSs)
\[
H^\ast(C_2;\pi_\star b(KU_2^\wedge))\Rightarrow \pi_{\star}b(KO_2^\wedge),\qquad H^\ast(\bbZ\{\psi^3\};\pi_\star b(KO_2^\wedge))\Rightarrow \pi_\star b(S_{K(1)}).
\]
These spectral sequences mostly determine the additive structure of $\pi_\star b(KO_2^\wedge)$ and $\pi_\star b(S_{K(1)})$, but hide all of the most interesting multiplicative structure. In fact, much of this additive structure is essentially classical. The groups $\pi_\star b(KO_2^\wedge)$ record the $KO_2^\wedge$-theory of stunted projective spectra, and the groups $\pi_\star b(S_{K(1)})$ may be read off these given their action by $\psi^3$, and both of these go back to the original work of Adams \cite{adams1962vector}. The groups $\pi_\star b(S_{K(1)})$ are also closely related to the $C_2$-equivariant $J$-homomorphism, which has been studied by several people \cite{loffler1977equivariant, crabb1980z2, minami1983equivariant}, as well as to the unstable $v_1$-periodic homotopy groups of spheres, as studied extensively by Davis, Mahowald, and others \cite{mahowald1982image, davismahowald1987homotopy}.

Thus the primary contribution of this paper is a complete description of the Green functor structure of $\ul{\pi}_\star b(S_{K(1)})$, including all products, restrictions, and transfers. Our main tools are the cofiber sequence $C_{2+}\rightarrow S^0\rightarrow S^\sigma$, James periodicity, and a good understanding of cell structures for stunted projective spectra and their $K(1)$-localizations.

\subsection{Conventions and notation}\label{ssec:conventions}

\subsubsection{Implicit localization}

Except where otherwise noted, everything we consider will be implicitly $K(1) = KU/(2)$-localized, so that for example $S = S_{K(1)}$ and $KO = KO_2^\wedge$.

\subsubsection{\texorpdfstring{$S$}{S}-duals}
We sometimes write $D(X) = F(X,S)$.

\subsubsection{Bigrading}

We find it most convenient to organize our computation by stem and coweight, and for this reason employ the grading convention
\[
\pi_{s,c} = \pi_{c+(s-c)\sigma}.
\]
This choice naturally arises from viewing a $C_2$-spectrum $X$ as sandwiched between its underlying spectrum and geometric fixed points, through the maps
\begin{center}
\begin{tikzcd}
\pi_{a+b}^eX&\ar[l,"\res"']\pi_{a+b\sigma}X\ar[r,"\Phi^{C_2}"]&\pi_a \Phi^{C_2}X
\end{tikzcd},
\end{center}
which with our conventions get reindexed to simply
\begin{center}
\begin{tikzcd}
\pi_s^e X&\ar[l,"\res"']\pi_{s,c}X\ar[r,"\Phi^{C_2}"]&\pi_c\Phi^{C_2}X
\end{tikzcd}.\end{center}
We will not make any direct use of geometric fixed points, but have found that the structure of $\pi_{\ast,\ast}b(S_{K(1)})$ is still most conveniently organized in this way.

\subsubsection{The element $\rho$}

There is an element $\rho\in\pi_{-1,0}S_{C_2}$ represented by the inclusion of poles $S^0\rightarrow S^\sigma$. This class plays an important role in $C_2$-equivariant homotopy theory. For example, $\rho$-completion models Borel completion and $\rho$-inversion models geometric fixed points, i.e.\ $X_\rho^\wedge\simeq F(EC_{2+},X)$ and $(X[\rho^{-1}])^{C_2} \simeq \Phi^{C_2}X$.

For our purposes, the most important aspect of $\rho$ is the cofiber sequence
\begin{center}\begin{tikzcd}
C_{2+}\ar[r]&S^0\ar[r,"\rho"]&S^\sigma
\end{tikzcd},$\qquad$ or dually,$\qquad$
\begin{tikzcd}
S^{-\sigma}\ar[r,"\rho"]&S^0\ar[r]&C_{2+}
\end{tikzcd}.
\end{center}
Smashing this with any $C_2$-spectrum $X$ yields a long exact sequence
\begin{center}\begin{tikzcd}
\cdots\ar[r]&\pi_{\ast}^e X\ar[r,"\tr"]&\pi_{\ast,c}^{C_2} X\ar[r,"\rho"]&\pi_{\ast-1,c}^{C_2}X\ar[r,"\res"]&\pi_{\ast-1}X\ar[r]&\cdots
\end{tikzcd},\end{center}
and thus
\begin{gather*}
\ker(\res\colon\pi_{\ast,c}X\rightarrow \pi_{\ast}^e X) = \im(\rho\colon \pi_{\ast+1,c}^{C_2}X\rightarrow \pi_{\ast,c}^{C_2} X)\\
\ker(\rho\colon \pi_{\ast,c}^{C_2}X\rightarrow \pi_{\ast-1,c}^{C_2}X) = \im(\tr\colon \pi_{\ast}^e X \rightarrow \pi_{\ast,c}^{C_2}X).
\end{gather*}
We will use these, particularly the first, throughout the paper.

\begin{rmk}\label{rmk:rhoborel}
If $X$ is an ordinary spectrum, then the action of 
\[
X^{-\ast} P_{w+1} = \pi_{w+1+\ast,\ast}b(X)  \xrightarrow{\rho}\pi_{w+\ast,\ast}b(X) = X^{-\ast} P_{w}
\]
is induced by the collapse map $P_{w}\rightarrow P_{w+1}$. In particular, $b(X)$ is $\rho$-complete, and if $
\Cof(\rho^{n+1}) = \Cof(\rho^{n+1}\colon S^{-(n+1)\sigma}\rightarrow S^0),
$
then
\[
\pi_{s,c}b(X)\otimes \Cof(\rho^{n+1}) \cong X^{-c-1}(P_{s-c}^{s-c+n})\cong X_{c}(P^{c-s-1}_{c-s-1-n}).
\]
Here, if $n\geq 0$ then $P_m^{m+n}$ is the stunted projective spectrum with one cell in each dimension from $m$ to $m+n$.
\tqed
\end{rmk}

\begin{rmk}
The class $\rho$ is sometimes denoted $a_\sigma$; our notation is borrowed from the $\bbR$-motivic context. Likewise, we will write $\tau$ for what would sometimes be denoted $u_\sigma^{-1}$.

Our computation of $\pi_{\ast,\ast}b(S_{K(1)})$ will place $\rho$ into a larger family of elements $\omega_a\in\pi_{8a-1,0}b(S_{K(1)})$, and in that context we will use $\omega_0$ and $\rho$ interchangeably.
\tqed
\end{rmk}

\section{Real \texorpdfstring{$K$}{K}-theory}

This section computes $\pi_{\ast,\ast}b(KO)$. The first step is to understand $\pi_{\ast,\ast}b(KU)$, and we begin in \cref{ssec:coc} by describing $\pi_{\ast,\ast}b(E)$ for more general complex-oriented theories $E$. We descend to $\pi_{\ast,\ast}b(KO)$ in \cref{ssec:ko} using the HFPSS associated to $KO\simeq KU^{\h C_2}$.

\subsection{The complex-oriented cohomology of stunted projective spectra}\label{ssec:coc}

In this subsection, we drop the assumption that everything is $K(1)$-local. Let $E$ be an even-periodic Landweber exact ring spectrum. We wish to describe $\pi_{\ast,\ast}b(E)$ together with its action by the Adams operations $\psi^k$.

Choose a periodic complex orientation $z\in E^0 BU(1)$, so that $E^\ast BU(1)\cong E_\ast[[z]]$ and $E_\ast \cong E_0[u^{\pm 1}]$, where $u\in \pi_2 E$ is obtained from $z$ by restricting to the bottom cell of $BU(1)$. For $k\in\bbZ$, write $[k](z) \in E_0[[z]]$ for the $k$-series of the formal group law associated to $E$. This satisfies $[k](z) = kz+O(z^2)$, and we write $\langle k \rangle(z) = \tfrac{[k](z)}{z}$.

\begin{lemma}\label{lem:bc2}
There is an isomorphism
\[
E^\ast BC_2\cong E_\ast[[z]]/([2](z)),
\]
and $\langle 2 \rangle(z) \in E^0 BC_2$ is the class of the transfer $BC_{2+}\rightarrow S^0\rightarrow E$.
\end{lemma}
\begin{proof}
This is well known, see for instance \cite[Lemma 5.7]{hopkinskuhnravenel2000generalized}; note that our $z$ relates to their $x$ by $z = u^{-1} x$, and that the proofs only require that $E$ is a complex-oriented theory with $[2](x)$ a non-zero-divisor in $E^\ast BU(1)$.
\end{proof}

Now write $\tau^{-2} \in \pi_{0,-2}b(E) = [\Sigma^{-2}P_2,E]$ for the Thom class of the complex bundle $2\sigma = \sigma\otimes\bbC$, so that $z = \rho^2\tau^{-2} u \in E^0 BC_2 = E^0 P_0$, and abbreviate $h = \langle 2 \rangle(z)$ so that $E^\ast BC_2 = E_\ast[[z]]/(z\cdot h)$.

\begin{prop}\label{prop:b}
There is an isomorphism
\[
\pi_{\ast,\ast}b(E) \cong E_\ast[\rho,\tau^{\pm 2}]_\rho^\wedge/(\rho\cdot h).
\]
\end{prop}
\begin{proof}
The relation $\rho\cdot h = 0$ holds in $\pi_\star S_{C_2}$, and this induces a map
\[
E_\ast[\rho,\tau^{\pm 2}]_\rho^\wedge/(\rho\cdot h)\rightarrow \pi_{\ast,\ast}b(E).
\]
As $\tau^2 \in \pi_{0,2}b(E)$ is invertible, to show that this map is an isomorphism we may reduce to considering bidegrees $(s,c)$ with $0\leq s-c \leq 1$. When $s-c = 0$, this amounts to the computation of $E^\ast BC_2 = E^\ast P_0$ given in \cref{lem:bc2}. The case $s-c=1$ follows from this and the observation that $\rho$ identifies $\pi_{1,0}b(E)$ as the augmentation ideal of $\pi_{0,0}b(E)$.
\end{proof}

\begin{rmk}
It is possible to give a coordinate-free formulation of \cref{prop:b}. Given a ring $R$ and invertible ideal $I\subset R$, define $R[I^{\pm 1}] = \bigoplus_{n\in\bbZ}I^{\otimes n}$. Given an ideal $J\subset I$, define the graded ring $A$ by $A_{2n} = I^{\otimes n}/(J\cdot I^{\otimes n})$ and $A_{2n-1} = I^{\otimes n}/((J\otimes I^{-1})\cdot I^{\otimes n})$. Thus $A_{2\ast} = R/J\otimes_R R[I^{\pm 1}]$, and the rest of $A$ records the fact that $J\subset I$. The obvious maps $A_n\rightarrow A_{n-1}$ make $A$ into an algebra over the graded polynomial ring $(R/J)[\rho]$, where $|\rho| = -1$.

Now let $R = E^0 BU(1)$ and $\bbG = \Spf R$. Let $I$ be the ideal of functions on $\bbG$ vanishing at $0$, so $R[I^{\pm 1}]$ deforms the ring of meromorphic functions on $\bbG$ \cite[Definition 5.20]{strickland1999formal}. Let $J\subset R$ be the ideal carving out the $2$-torsion points $\bbG[2]\subset\bbG$. Then $J\subset I$ as $0\in\bbG[2]$, and applying the above construction returns $\pi_{\ast,0}b(E)$.
\tqed
\end{rmk}

Because $E$ is Landweber exact, every automorphism of the formal group $\bbG = \Spf E^0 BU(1)$ induces a multiplicative automorphism of $E$. In particular, given $k\in \bbZ$ invertible in $E_0$, formal multiplication by $k$ is an automorphism of $\bbG$, and this gives rise to the Adams operation $\psi^k\colon E\rightarrow E$. If $E$ is $2$-complete, then this extends to any $2$-adic unit $k\in\bbZ_2^\times$.

\begin{prop}\label{lem:adams}
The Adams operations $\psi^k$ act by $E_0$-algebra maps, and their action on $\pi_{\ast,\ast}b(E)$ is determined by
\[
\psi^k(u^j) = k^j u^j,\qquad \psi^k(\tau^{2j}) = \tau^{2j}(1+\tfrac{1}{2}(k^j-1)(2-h)),\qquad\psi^k(\rho) = \rho.
\]
\end{prop}
\begin{proof}
That $\psi^k(u^j) = k^j u^j$ is standard, and $\psi^k(\rho) = \rho$ as $\rho$ is in the Hurewicz image, so it suffices to compute $\psi^k(\tau^{2j})$. By multiplicativity, using the relation $h^2 = 2h$, one easily reduces to the case $j = -1$. Observe that the map $\Sigma^{-2}P_{2}\rightarrow\Sigma^{\infty-2}BU(1)$ acts in $E^0$-cohomology by
\[
u^{-1}zE_0[[z]]\rightarrow E^2 P_2,\qquad u^{-1}zf(z)\mapsto \tau^{-2} f(z).
\]
The operation $\psi^k$ is defined so that $\psi^k(z) = [k](z)$ in $E^0 BU(1)$, and we deduce that
\[
\psi^k(\tau^{-2}) = k^{-1}\tau^{-2}\cdot \langle k\rangle(z)
\]
in $\pi_{0,-2}b(E) = E^2 P_2$. Now abbreviate $p_k = k + \tfrac{1}{2}(1-k)(2-h)$. We claim that $\langle k \rangle (z) \equiv p_k\pmod{[2](z)}$. To prove this, we proceed as in \cite[Remark 6.15]{hopkinskuhnravenel2000generalized}. As $k$ is odd, we have $[k](z)\equiv z \pmod{[2](z)}$. Thus
\begin{align*}
z\cdot (\langle k \rangle(z) - 1) \equiv 0 \pmod{[2](z)},&\qquad \langle k \rangle(z) - 1 = (k-1)+O(z), \\
z\cdot (p_k-1) \equiv 0 \pmod{[2](z)},&\qquad p_k-1 = (k-1) + O(z),
\end{align*}
the second row following from the relations $z h = 0$ and $h = 2 + O(z)$. As the kernel of multiplication by $z$ on $E_0[[z]]/([2](z))$ is the free $E_0$-module generated by $h$, any $f(z)\in E_0[[z]]$ satisfying $z f(z)\equiv 0 \pmod{[2](z)}$ is determined modulo $[2](z)$ by its constant coefficient, and so the above identities imply that $\langle k \rangle(z) - 1 \equiv p_k-1\pmod{[2](z)}$. Thus
\begin{align*}
\psi^k(\tau^{-2}) &= k^{-1}\tau^{-2}\cdot \langle k \rangle(z), \\
&= k^{-1}\tau^{-2}\cdot (k+\tfrac{1}{2}(1-k)(2-h)) = \tau^{-2}(1+\tfrac{1}{2}(k^{-1}-1)(2-h))
\end{align*}
in $\pi_{0,-2}b(E)$ as claimed.
\end{proof}

\subsection{The homotopy fixed point spectral sequence}\label{ssec:ko}

We now return to the $K(1)$-local setting, specializing the above discussion to $KU = KU_2^\wedge$. We consider $KU$ to be oriented with formal group law $x+y-xy$, so that $[2](z) = 2z-z^2$, and write $\beta \in \pi_2 KU$ for the associated Bott element. 

\begin{lemma}\label{lem:bku}
There is an isomorphism
\[
\pi_{\ast,\ast}b(KU) = \bbZ_2[\beta^{\pm 1},\tau^{\pm 2},\rho]/(\rho\cdot h),\qquad h = 2 - \rho^2\tau^{-2}\beta.
\]
The Adams operations $\psi^k$ for $k\in\bbZ_2^\times$ act on $\pi_{\ast,\ast}b(KU)$ by multiplicative automorphisms, and are determined by
\[
\psi^k(\beta^j) = k^j \beta^j,\qquad \psi^k(\tau^{2j}) = \tau^{2j}(1+\tfrac{1}{2}(k^j-1)\cdot\rho^2\tau^{-2}\beta),\qquad\psi^k(\rho) = \rho.
\]
\end{lemma}
\begin{proof}
This is a direct specialization of \cref{prop:b} and \cref{lem:adams}.
\end{proof}

\cref{lem:bku} is pictured in the following:

\begin{center}
$\pi_{\ast,\ast}b(KU)$
\includegraphics[page=1,width=\textwidth]{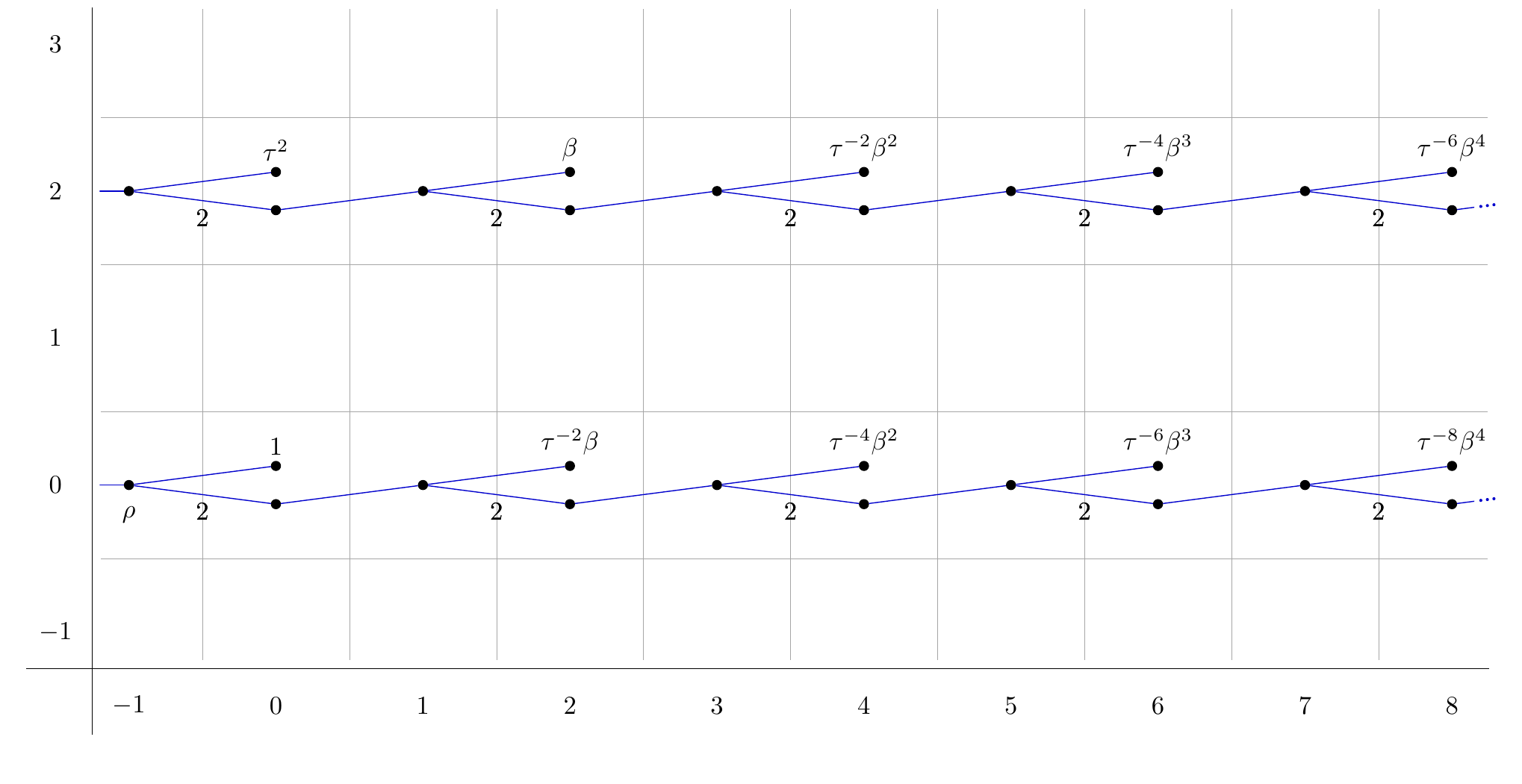}
\end{center}
\vspace{-.25cm}
Each black dot is a copy of $\bbZ_2$, with coordinate $(x,y)$ corresponding to the group $\pi_{x,y}b(KU)$. The blue lines point to the left, and depict multiplication by $\rho$.

We now begin the descent to $b(KO)$. The involution $\psi^{-1}$ is associated to a coherent action of $C_2$ on $KU$ satisfying $KU^{\h C_2}\simeq KO$. As $b(\bs) = F(EC_{2+},i_\ast(\bs))$ preserves limits, we have $b(KU)^{\h C_2}\simeq b(KU^{\h C_2})\simeq b(KO)$, and so there is an HFPSS
\[
E_2 = H^\ast(C_2;\pi_{\ast,\ast}b(KU))\Rightarrow\pi_{\ast,\ast}b(KO)
\]
combining all the individual HFPSSs for each $KO^\ast P_w$. Write
\[
H^1(C_2;\pi_2 KU) = \bbZ/(2)\{\mu\},\qquad \eta_0 = \rho\tau^{-2}\beta \in \pi_\sigma b(KU).
\]
Thus $\mu$ will survive to the nonequivariant Hopf map in $\pi_1 KO \subset \pi_{1,1}b(KO)$, and $\eta_0$ is up to a sign the restriction of the transfer $\Sigma^\infty_+ BC_2\rightarrow S^0\rightarrow KU$ to $P_1 \simeq \Sigma^\infty BC_2$. In particular, $\eta_0$ is in the image of the Hurewicz map $\pi_{\ast,\ast}S_{C_2}\rightarrow \pi_{\ast,\ast}b(KU)$.

\begin{rmk}\label{rmk:eta}
One may identify $\eta_0$ as the Hurewicz image of $-\eta_{C_2}$, where $\eta_{C_2}$ is the $C_2$-equivariant Hopf map defined with conventions as in \cite{guillouisaksen2019bredonlandweber}. The necessary sign is most easily detected by comparing the relations $\rho^2\eta_0 = 2\rho$ and $\rho^2\eta_{C_2} = -2\rho$. Likewise one may identify $-\eta_{C_2} = \hat{\eta}$, where $\hat{\eta}$ is as defined in \cite{arakiiriye1982equivariant}.
\tqed
\end{rmk}

\begin{lemma}\label{prop:e2}
$H^\ast(C_2;\pi_{\ast,\ast}b(KU))\cong\bbZ_2[\beta^{\pm 2},\tau^{\pm 4},\rho,\tau^2 h,\eta_0,\mu]/I$, where
\begin{gather*}
I = \left(\begin{array}{c}
\eta_0^2 - \rho^2\tau^{-4}\beta^2,\quad\rho^2 \eta_0 - 2 \rho,\quad\rho \eta_0^2 - 2 \eta_0,\quad\tau^2 h \cdot \tau^2 h - 2\tau^4 (2 - \rho \eta_0),\\
\rho\cdot\tau^2 h,\quad\eta_0\cdot\tau^2 h,\quad 2\mu,\quad\mu\cdot\tau^2 h,\quad \mu\cdot\rho^2
\end{array}\right).
\end{gather*}
The residual action of $\psi^k$ for $k\in\bbZ_2^\times/\{\pm 1\}$ is given by
\begin{gather*}
\psi^k(\beta^{2j}) = k^{2j} \beta^2,\qquad \psi^k(\tau^{4j}) = \tau^{4j}(1+\tfrac{1}{2}(k^{2j}-1)\rho\eta_0),\\
\psi^k(\rho) = \rho,\qquad \psi^k(\tau^2h) = \tau^2h,\qquad \psi^k(\eta_0) = \eta_0,\qquad\psi^k(\mu) = \mu.
\end{gather*}
\end{lemma}
\begin{proof}
This follows by a direct computation from \cref{lem:bku}.
\end{proof}

Abbreviate $\per = \tau^{-8}\beta^4 \in \pi_{8,0}b(KU)$, and note that $\sqrt{\per}\in H^0(C_2;\pi_{4,0}b(KU))$. \cref{prop:e2} is now pictured in the following:

\begin{center}
$H^\ast(C_2;\pi_{\ast,\ast}b(KU))$
\includegraphics[page=1,width=\textwidth]{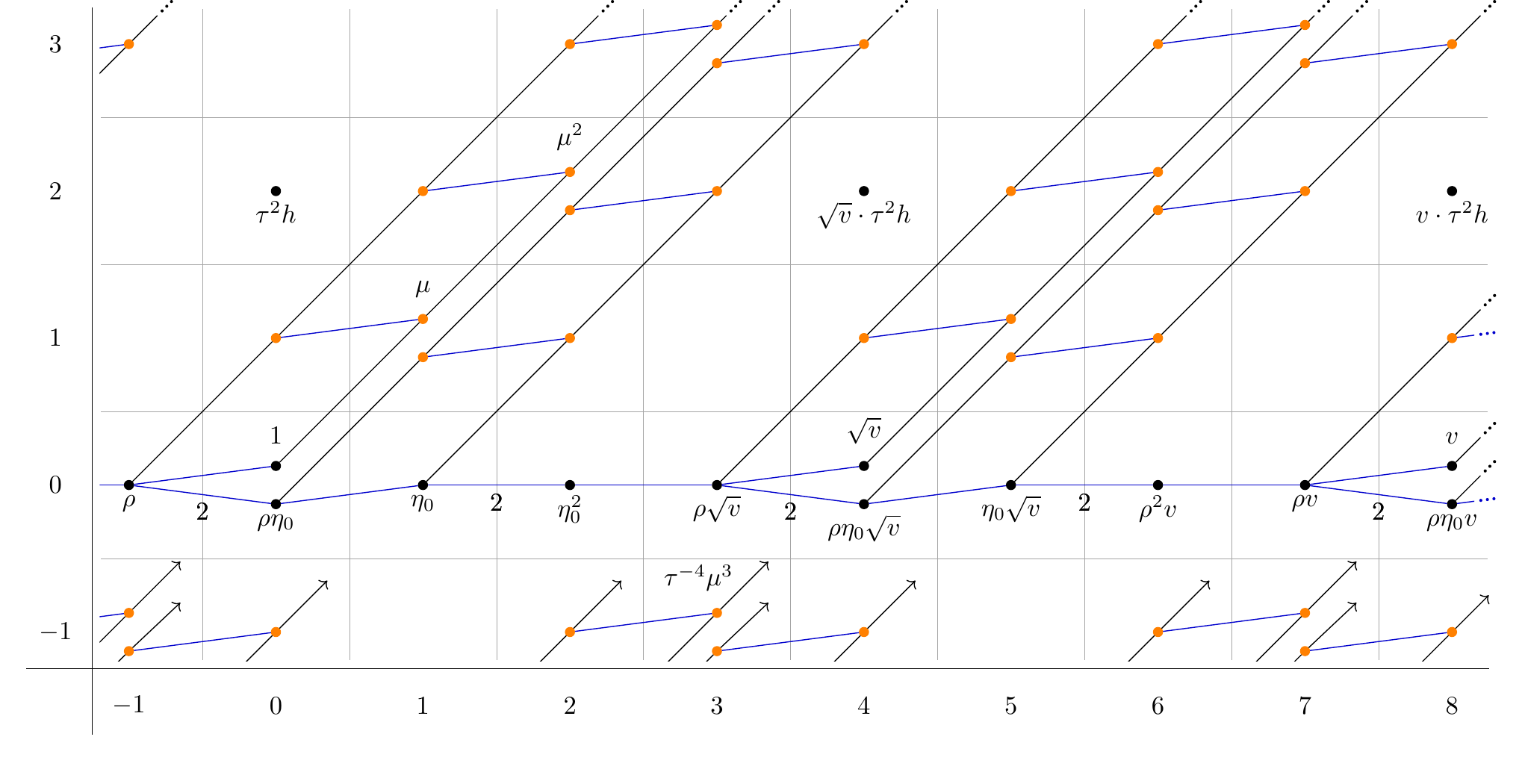}
\end{center}
\vspace{-.25cm}
Each black dot is a copy of $\bbZ_2$, each orange dot is a copy of $\bbZ/(2)$, and the blue lines point to the left and depict multiplication by $\rho$. We have truncated all $\mu$-towers at $\mu^3$, so that coordinate $(x,y)$ corresponds to $H^n(C_2;\pi_{x+n,y+n}b(KU))$ for $0\leq n \leq 3$.

\begin{lemma}\label{prop:d3}
Differentials in the HFPSS for $\pi_{\ast,\ast}b(KO)$ are generated by
\[
d_3(\beta^2) = \mu^3,\quad d_3(\mu) = 0,\quad d_3(\rho) = 0,\quad d_3(\tau^2 h) = 0,\quad d_3(\eta_0) = 0,\quad d_3(\tau^4) = 0.
\]
\end{lemma}
\begin{proof}
The differential $d_3(\beta^2) = \mu^3$ is classical. The differentials $d_3(\rho) = 0$ and $d_3(\eta_0) = 0$ hold as these classes are in the Hurewicz image, and $d_3(\tau^2h) = 0$ as there is no possible target. Finally, $d_3(\tau^4) = 0$ as $4\sigma$ is $KO$-orientable.
\end{proof}

The classes $\tau^4\in\pi_{0,4}b(KU)$ and $v\in \pi_{8,0}b(KU)$ survive to give $\pi_{\ast,\ast}b(KO)$ a $4$-fold vertical periodicity and $8$-fold horizontal periodicity. These groups are depicted in the following:
\begin{center}
$\pi_{\ast,\ast}b(KO)$
\includegraphics[page=1,width=\textwidth]{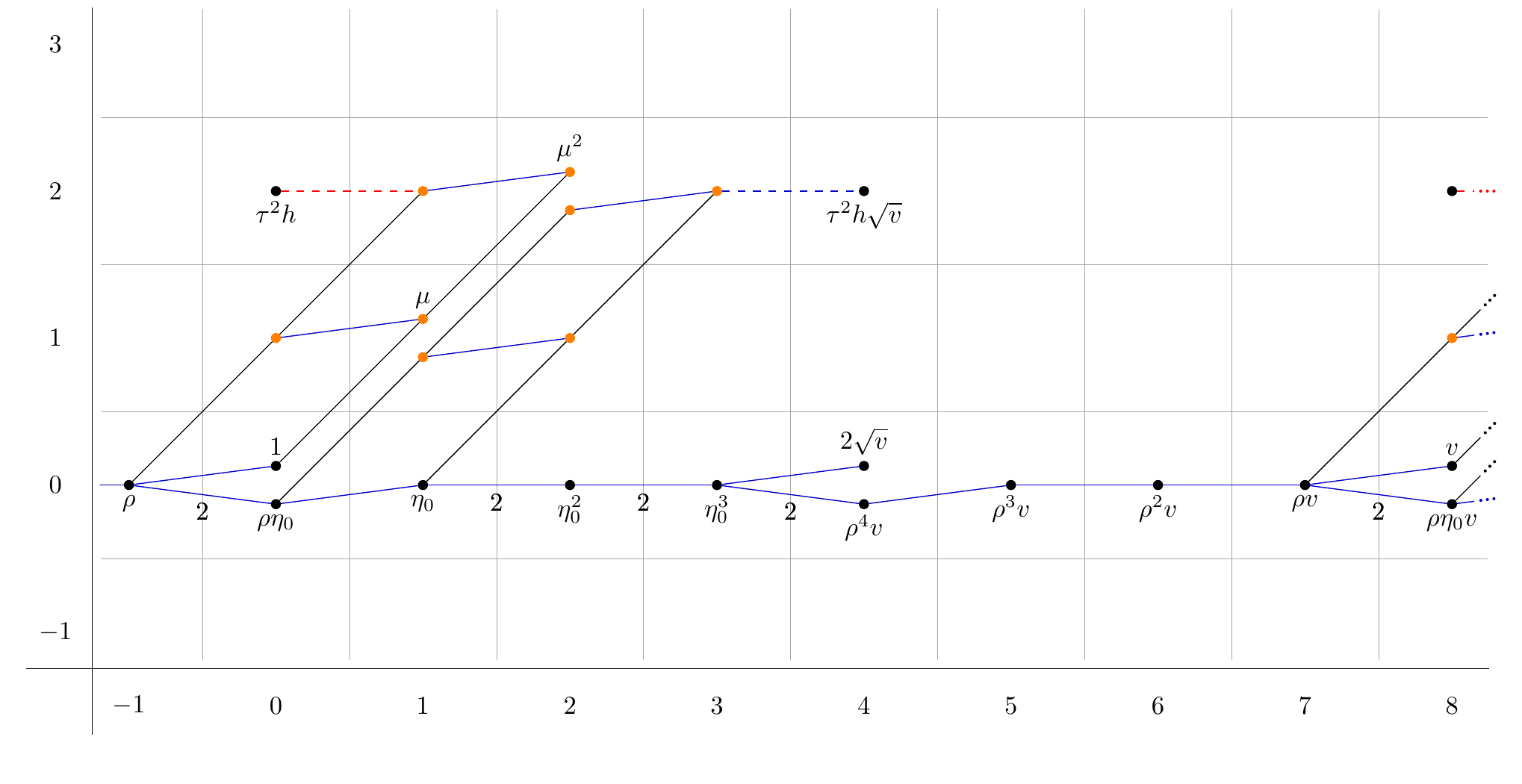}
\end{center}
The dashed red and blue lines are hidden $\eta_0$ and $\rho$-extensions, which we now resolve.

\begin{lemma}\label{lem:koext2}
$
\rho\cdot\tau^2h \sqrt{v}= \eta_0\mu^2.
$
\end{lemma}
\begin{proof}
Note that $\eta_0\mu^2\in\pi_{3,2}b(KO)$. As $\pi_3 KO = 0$, necessarily $\eta_0\mu^2$ is in the kernel of the forgetful map $\pi_{3,2}b(KO)\rightarrow \pi_3 KO$. Thus $\eta_0\mu^2$ is in the image of $\rho$, and the indicated relation is the only possibility.
\end{proof}

\begin{lemma}\label{lem:koext1}
$
\eta_0\cdot\tau^{2}h = \rho\mu^2.
$
\end{lemma}
\begin{proof}
Following \cite[Proposition 10.13]{guillouhillisaksenravenel2020cohomology}, there is a Wood-type cofibering
\begin{center}\begin{tikzcd}
\Sigma^\sigma b(KO)\ar[r,"\eta_0"]&b(KO)\ar[r]&K\bbR
\end{tikzcd},\end{center}
where $K\bbR$ is Atiyah's Real $K$-theory. As $K\bbR$ is $(2,1)$-periodic, we have $\pi_{1,2}K\bbR \cong \pi_{3,3}K\bbR = \pi_3 KO = 0$. Thus $\rho\mu^2 \mapsto 0$ in $\pi_{1,2}K\bbR$, implying that $\rho\mu^2$ is in the image of $\eta_0$, and the indicated relation is the only possibility.
\end{proof}

\begin{rmk}
One may also compute $\pi_{\ast,\ast}b(ko)$ via the Adams spectral sequence
\[
\Ext_{\calA(1)}(\bbF_2;H^\ast P_\ast)\Rightarrow ko^\ast P_\ast = \pi_{\ast,\ast}b(ko).
\]
This spectral sequence collapses with no differentials, and the relations $\rho\cdot\tau^2h\sqrt{v} = \eta_0\mu^2$ and $\eta_0\cdot \tau^2 h = \rho \mu^2$ are seen on the $E_2$-page. We leave the details to the interested reader.
\tqed
\end{rmk}

We also record the following here for later use.

\begin{lemma}\label{lem:koforget}
The forgetful map $\pi_{1,0}b(KO)\rightarrow\pi_1 KO$ sends $\eta_0\mapsto\mu$.
\end{lemma}
\begin{proof}
As $\eta_0$ is not in the image of $\rho$, it is not in the kernel of the forgetful map. As $\pi_1 KO = \bbZ/(2)\{\mu\}$, it must be that $\eta_0\mapsto\mu$.
\end{proof}

We may summarize our computation with the following.

\begin{prop}
$\pi_{\ast,\ast}b(KO) = \bbZ_2[v^{\pm 1},\tau^{\pm 4},\rho,\eta_0,\tau^2 h,\tau^2 h\sqrt{v},2\sqrt{v},\mu]/I$, where $I$ is generated by the following relations:
\begin{enumerate}
\item Relations involving $\mu$:
\begin{gather*}
\mu^3 = 0,\quad
2\mu = 0,\quad
\mu\cdot 2\sqrt{v} = 0,\\
\mu\cdot \tau^2h = 0,\quad
\mu\cdot\tau^2h\sqrt{v} = 0,\quad
\mu\cdot\rho^2 = 0, \quad
\mu\cdot\eta_0^2 = 0;
\end{gather*}
\item Relations describing how $\rho$ and $\eta_0$ interact:
\[
\rho^2 \eta_0 = 2 \rho, \quad
\rho \eta_0^2 = 2 \eta_0, \quad
\eta_0^4 = \rho^4 v, \quad
\rho\cdot 2 \sqrt{v} = \eta_0^3;
\]
\item Relations lifting $\rho\cdot h = 0$: 
\[
\rho\cdot\tau^2h = 0,\quad\eta_0\cdot\tau^2h\sqrt{v} = 0;
\]
\item Hidden extensions lifting $\rho\cdot  h = 0$:
\[
\eta_0\cdot\tau^2 h = \rho \mu^2, \quad
\rho\cdot\tau^2 h \sqrt{v} =  \eta_0 \mu^2;
\]
\item Relations lifting $h^2 = 2 h$:
\begin{gather*}
\tau^2h\cdot\tau^2 h = 2\tau^4(2-\rho\eta_0),\quad
\tau^2h\cdot\tau^2h\sqrt{v} = 2\sqrt{v}\cdot\tau^4(2-\rho\eta_0),\\
\tau^2h\sqrt{v} \cdot\tau^2h\sqrt{v} = 2 v\tau^4(2-\rho\eta_0);
\end{gather*}
\item Relations implicit in the notation:
\[ 
\tau^2h\cdot 2\sqrt{v} = 2 \cdot\tau^2h\sqrt{v}, \quad
\tau^2h\sqrt{v}\cdot 2 \sqrt{v} = 2 v \cdot \tau^2h,\quad
2\sqrt{v}\cdot 2\sqrt{v} = 4 v.
\]
\end{enumerate}
\end{prop}
\begin{proof}
This summarizes the computation carried out above.
\end{proof}

\section{Some preliminaries}

This section sets up some tools that we will need in our computation of $\ul{\pi}_{\ast,\ast}b(S)$. In \cref{ssec:cell}, we use our computation of $\pi_{\ast,\ast}b(KO)$ to deduce information about the $K(1)$-local spectra $P_j$. This partially relies on \cref{ssec:classic}, where we recall the homotopy ring of the nonequivariant $K(1)$-local sphere. Finally, in \cref{ssec:james} we recall how James periodicity manifests as $\tau$-periodicity in $C_2$-equivariant stable homotopy theory.

\subsection{\texorpdfstring{$K(1)$}{K(1)}-local cell structures}
\label{ssec:cell}

The $K(1)$-local Picard group at $p=2$ carries a unique exotic element, i.e.\ there is up to equivalence a unique $K(1)$-local spectrum $T$ such that $KU_\ast T \simeq KU_\ast$ as $\bbZ_2^\times$-modules despite $T\not\simeq S$. This is shown in \cite{hopkinsmahowaldsadofsky1994constructions}, and much of the following may be found there in some form as well.

The spectrum $T$ can be detected using $KO$-theory, through an equivalence $KO\otimes T\simeq \Sigma^4 KO$. In practical terms, this means the following. Let $L$ be a spectrum satisfying $KU_\ast L \cong KU_\ast$ as $\bbZ_2^\times$-modules, so that either $L\simeq S$ or $L\simeq T$. Choose a generator $\iota\in KU_0 L$, and consider the HFPSS
\[
H^\ast(C_2;KU_\ast L) \cong \bbZ_2[\beta^{\pm 2},\mu]/(2\mu)\otimes\bbZ\{\iota\}\Rightarrow KO_\ast L.
\]
In this spectral sequence, one must have either $d_3(\iota) = 0$, in which case $L\simeq S$, or else $d_3(\iota) = \mu^3\beta^{2}\iota$, in which case $L\simeq T$. Applied to the spectra $P_{2n+1}$, this reveals the following.

\begin{lemma}\label{prop:pic}
There are equivalences
\[
P_{2n+1} \simeq
\begin{cases}
 S & n \equiv 0,3\pmod{4}, \\
 T & n \equiv 1,2\pmod{4}.
\end{cases}
\]
\end{lemma}
\begin{proof}
The element
\[
\rho\beta^{n+1}\tau^{-2(n+1)} \in \pi_{2n+1,0}b(KU) = KU^0 P_{2n+1}
\]
gives a generator of $KU^\ast P_{2n+1}\cong KU_\ast$ fixed under the action of $\bbZ_2^\times$. Thus $P_{2n+1}\simeq S$ when $\rho\beta^{n+1}\tau^{-2(n+1)}$ is a permanent cycle in the HFPSS of \cref{ssec:ko}, and otherwise $P_{2n+1}\simeq T$. The lemma now follows from \cref{prop:d3}.
\end{proof}

To go further, we require some knowledge of $\pi_\ast S$ and $\pi_\ast T$. We will recall these in \cref{ssec:classic}, but give the following now.

\begin{prop}\label{prop:cell}
The spectrum $P_{2n}$ participates in cofiberings
\begin{gather*}
 S^{2n}\rightarrow P_{2n}\rightarrow P_{2n+1}\\
 S^{2n-1}\rightarrow P_{2n-1}\rightarrow P_{2n}
\end{gather*}
realizing $P_{2n}$ as a $2$-cell complex in the $K(1)$-local category in two ways, where we consider the Picard element $T$ to be a single cell. The attaching map vanishes if $n = 0$, and is otherwise a nonzero simple $2$-torsion map. Explicitly, if we choose equivalences $P_{2n+1}\simeq S$ or $T$, then the attaching map may be regarded as an element of $\pi_\ast S$ or $\pi_\ast T$, and with notation from \cref{ssec:classic} it is as described in the following tables.

\begin{center}
\begin{tabular}{ll}
\multicolumn{2}{c}{First cofibering}\\
\midrule
$n$ & Attaching map\\
\midrule
$4c$&$2^{j_{-2c}-1}\rho_{-c}$\\
$4c-1$&$(\text{unit})\cdot\mu_{-c}$\\
$4c-2$&$4\rho_{-c+1/2}$\\
$4c-3$&$(\text{unit})\cdot\mu_{-c+1/2}$
\end{tabular}
$\qquad\qquad$
\begin{tabular}{ll}
\multicolumn{2}{c}{Second cofibering}\\
\midrule
$n$ & Attaching map \\
\midrule
$4c$&$2^{j_{2c}-1}\rho_{c}$\\
$4c+1$&$(\text{unit})\cdot\mu_{c}$\\
$4c+2$&$4\rho_{c+1/2}$\\
$4c+3$&$(\text{unit})\cdot\mu_{c+1/2}$
\end{tabular}
\end{center}
\end{prop}
\begin{proof}
These cofiberings exist before $K(1)$-localization, and \cref{prop:pic} shows that $K(1)$-locally they exhibit $P_{2n}$ as a $2$-cell complex in the stated sense. The attaching map in either case may therefore be viewed as an element in $\pi_\ast S$ or $\pi_\ast T$. 

If $n$ is odd, then our computation of $\pi_{\ast,\ast}b(KO)$ shows that $F(P_{2n},KO)\simeq KU$, and so the attaching map is detected in $KO$-theory. The listed classes are the only possibilities in their degrees; the particular unit cannot be identified as it is sensitive to the choice of equivalence $P_{2n+1}\simeq S$ or $T$.

If $n$ is even, then our computation of $\pi_{\ast,\ast}b(KO)$ shows that the cofiberings induce short exact sequences
\begin{center}\begin{tikzcd}[column sep=small, row sep=tiny]
0\ar[r]&KO^\ast P_{2n+1}\ar[r]&KO^\ast P_{2n}\ar[r]&KO^\ast S^{2n}\ar[r]&0\\
0\ar[r]&KO^\ast S^{2n}\ar[r]&KO^\ast P_{2n}\ar[r]&KO^\ast P_{2n-1}\ar[r]&0
\end{tikzcd}.\end{center}
These are extensions of $KO_\ast\hyp\bbZ_2^\times/\{\pm 1\}$-modules, the $\Ext$ groups classifying such extensions may be identified as the cohomology groups $\smash{H^1_c(\bbZ_2^\times/\{\pm 1\};KO_\ast)}\cong H^1(\bbZ\{\psi^3\};KO_\ast)$ used to compute $\pi_\ast S$ and $\pi_\ast T$, and the cohomology class corresponding to these extensions is the $2$-complete analogue of the $e$-invariant \cite{adams1966on} of the attaching map. Thus the behavior of the attaching map is entirely governed by this extension. The action of the Adams operations on $KO^\ast P_{2n}$ reveals that this extension is trivial if $n=0$, and nontrivial but simple $2$-torsion if $n\neq 0$, and the listed class is the only possibility in its degree.
\end{proof}

\begin{rmk}
The two $2$-cell structures on $P_{2n}$ given in \cref{prop:cell} are an example of a fundamental duality present in $K(n)$-local equivariant homotopy theory. As the Tate construction vanishes $K(1)$-locally \cite{hoveysadofsky1996tate,greenleessadofsky1996tate}, the norm
\[
P_n = (S^{n\sigma})_{\h C_2}\rightarrow (S^{n\sigma})^{\h C_2}\simeq D((S^{-n\sigma})_{\h C_2}) = D(P_{-n})
\]
is an equivalence. Thus, dualizing the cofiber sequence 
\[
S^{2n}\rightarrow P_{2n}\rightarrow P_{2n+1}
\]
yields a cofiber sequence of the form
\[
S^{-2n}\leftarrow P_{-2n}\leftarrow P_{-2n-1}.
\]
Considering these sequences for both $n$ and $-n$ yields the two $2$-cell structures on $P_{2n}$.
\tqed
\end{rmk}

\subsection{The nonequivariant \texorpdfstring{$K(1)$}{K(1)}-local sphere}\label{ssec:classic}

We now recall the structure of and fix notation for $\pi_\ast S$ and $\pi_\ast T$, as has already been used in \cref{prop:cell} and will be needed in our computation of $\pi_{\ast,\ast}b(S)$. For $a\in\bbZ$ and $b\in\tfrac{1}{2}\bbZ$, abbreviate
\[
S_a = \pi_a S,\qquad S_{a,b} = \pi_a T^{\otimes 2b}.
\]
As $T\otimes T\simeq S$, the group $S_{a,b}$ only depends on the class of $b$ in $\tfrac{1}{2}\bbZ/\bbZ\cong\bbZ/(2)$. The notation $S_{\ast,\ast}$ will only be used in this subsection.

Define $j_a$ to be three more than the $2$-adic valuation of $a$, i.e.\ so that $\bbZ_2/(2^{j_a})\cong \bbZ_2/(8a) \cong \bbZ_2/(3^{2a}-1)$. When $a = 0$, we take $j_a = \infty$ with the understanding that $2^\infty = 0$ in $\bbZ_2$.

\begin{lemma}\label{lem:sgroups}
For $a\in \bbZ$ and $\epsilon\in \{0,\tfrac{1}{2}\}$, there are elements
\begin{gather*}
1\in S_{0,0},\qquad g\in S_{0,1/2},\\
\rho_{a+\epsilon}\in S_{8(a+\epsilon)-1,\epsilon},\qquad \mu_{a+\epsilon}\in S_{8(a+\epsilon)+1,a+\epsilon},\qquad \xi_{a+\epsilon}\in S_{8(a+\epsilon)+3,\epsilon}.
\end{gather*}
The groups $S_{\ast,\ast}$ are given by
\begin{gather*}
S_{0,0} = \bbZ_2\{1\}\oplus\bbZ/(2)\{\mu_0\rho_0\},\qquad S_{0,1/2} = \bbZ_2\{g\},
\end{gather*}
and otherwise
\begin{gather*}
S_{8(a+\epsilon)-1,\epsilon} = \bbZ_2/(2^{j_{2(a+\epsilon)}})\{\rho_{a+\epsilon}\},\qquad S_{8(a+\epsilon),\epsilon} = \bbZ/(2)\{\mu_0\rho_{a+\epsilon}\}, \\
S_{8(a+\epsilon)+1,\epsilon} = \bbZ/(2)\{\mu_{a+\epsilon},\mu_0^2\rho_{a+\epsilon}\},\qquad S_{8(a+\epsilon)+2,\epsilon} = \bbZ/(2)\{\mu_0\mu_{a+\epsilon}\}, \\
 S_{8(a+\epsilon)+3,\epsilon} = \bbZ_2/(2^{j_{2(a+\epsilon)+1}})\{\xi_{a+\epsilon}\},
\end{gather*}
all other groups being zero. In addition, there are relations
\begin{gather*}
\mu_{a+b}\mu_c = \mu_a\mu_{b+c},\qquad \mu_{a+b}\rho_c = \mu_a\rho_{b+c} ,\qquad \mu_a\mu_b\mu_c = 2^{j_{2(a+b+c)+1}-1}\xi_{a+b+c}, \\
\rho_a\xi_b = 0,\qquad g^2 = 4,\qquad g \xi_a = 2\rho_{a+1/2},\qquad g\rho_a = 2\xi_{a-1/2},\qquad g\mu_a = 0,
\end{gather*}
for $a,b,c\in \frac{1}{2}\bbZ$.
\end{lemma}
\begin{proof}
This is classical, so let us just briefly indicate the proof. The structure of $S_{\ast,\epsilon}$ may be computed via the HFPSSs
\[
H^\ast(C_2;KU_\ast T^{\otimes \epsilon})\Rightarrow KO_\ast T^{\otimes\epsilon},\qquad H^\ast(\bbZ\{\psi^3\};KO_\ast T^{\otimes\epsilon})\Rightarrow \pi_\ast T^{\otimes\epsilon}.
\]
See the beginning of \cref{ssec:cell} for a description of the first HFPSS. Most of the structure of $S_{\ast,\ast}$ may be immediately read off of these, only one must rule out the existence of hidden additive extensions and produce the hidden multiplicative extension $\mu_a\mu_b\mu_c = 2^{j_{2(a+b+c)+1}-1}\xi_{a+b+c}$. We cite \cite[Theorem 8.15]{ravenel1984localization} for these facts in the case of $S_\ast = S_{\ast,0}$. The corresponding facts involving $T$ are analogous; as they will not be needed in our computation of $\pi_{\ast,\ast}b(S)$, we omit the details.
\end{proof}

There is one additional subtlety that we must address. \cref{lem:sgroups} asserts the existence of classes $\mu_a \in S_{8a+1}$, and the proof further reveals that $\mu_a$ should be some element detected by $\beta^{4a}\mu\in \pi_{8a+1}KO$. However, if $\mu_a$ is such an element, then so is $\mu_a\cdot (1+\mu_0\rho_0)$, and it is difficult to distinguish between these two elements by purely formal means.

We resolve this issue in the usual way, following \cite[Theorem 12.13]{adams1966on}. There is a canonical choice for $\mu_0$, given by the Hurewicz image of the nonequivariant Hopf map $\eta_\cl$. We now consider as fixed throughout the paper some $\beta^4$-self map $\Sigma^8 S/(2)\rightarrow S/(2)$. This induces a secondary operation $\beta^4\colon \pi_n\rightharpoonup\pi_{n+8}$ defined on $2$-torsion elements, which can be regarded as the Toda bracket $\langle 8\sigma,2,\bs\rangle$ formed with respect to a specified nullhomotopy of $8\sigma\cdot 2$. In particular its indeterminacy is divisible by $2$, and so we may unambiguously define $\mu_a = \beta^{4a}\mu_0$.

\subsection{James periodicity}\label{ssec:james}

In this subsection, we drop the assumption that everything is $K(1)$-local. We have seen that $b(KO_2^\wedge)$ is $\tau^4$-periodic, and will see that much of this is visible in $b(S_{K(1)})$. This $\tau$-periodicity is a manifestation of James periodicity \cite{james1958cross}: if
\[
\gamma(n) = |\{k \in\bbZ : 0 < k \leq n,\, k\equiv 0,1,2,4\pmod{8}\}|
\]
is the $n$th Radon--Hurwitz number, then there are equivalences
\[
P_m^{m+n}\simeq \Sigma^{-2^{\gamma(n)}}P_{m+2^{\gamma(m)}}^{m+n+2^{\gamma(n)}}.
\]
This itself can be seen as a consequence of the fact that $2^{\gamma(n)}\rho\eta_0$ is divisible by $\rho^{n+1}$ in $\pi_{\ast,\ast}b(KO)$, see for instance \cite{mahowald1965short}. To make use of this periodicity in our setting, it is convenient to use the following refinement.

\begin{lemma}
There are equivalences $\tau^{2^{\gamma(n)}}\colon \Sigma^{0,2^{\gamma(n)}}\Cof(\rho^{n+1})\simeq\Cof(\rho^{n+1})$.
\end{lemma}
\begin{proof}
This is \cite[Theorem 7.7]{behrensshah2020c2}.
\end{proof}

We fix such equivalences, and define $\tau^{k2^{\gamma(n)}}$ for $k\in\bbZ$ by composition. After possibly adjusting these equivalences by a unit, we may suppose that $\tau^{2^{\gamma(n)}}$ acts as multiplication by $\tau^{2^{\gamma(n)}}$ on $b(KU_2^\wedge) \otimes \Cof(\rho^{n+1})$. Self maps such as these give rise to secondary operations in the usual way, which we now recall.

\begin{construction}\label{prop:tau}
Let $X$ be a $C_2$-spectrum. Then there are secondary operations
\[
\tau^{k2^{\gamma(n)}}\colon \pi_{s,c}X\rightharpoonup \pi_{s,c+k2^{\gamma(n)}}X,
\]
defined on the kernel of $\rho^{n+1}$ and with indeterminacy contained in the image of multiplication by an element of $\pi_{n,k2^{\gamma(n)}-1}S_{C_2}$, defined as follows. Fix $\alpha\in \pi_{s,c}X$ with $\rho^{n+1}\alpha = 0$, and choose an extension of $\alpha\colon \Sigma^{s,c}S_{C_2}\rightarrow X$ to a map $\alpha'\colon \Sigma^{s,c}\Cof(\rho^{n+1})\rightarrow X$. Now define $\tau^{k2^{\gamma(n)}}\alpha$ to be the composite 
\begin{center}\begin{tikzcd}
\Sigma^{s,c+k2^{\gamma(n)}}S_{C_2}\ar[r]&\Sigma^{s,c+k2^{\gamma(n)}}\Cof(\rho^{n+1})\ar[r,"\tau^{k2^{\gamma(n)}}","\simeq"']& \Sigma^{s,c}\Cof(\rho^{n+1})\ar[r,"\alpha'"]& X
\end{tikzcd}.\end{center}
The indeterminacy arising from our choice of $\alpha'$ is contained in the image of multiplication by the map
\begin{center}\begin{tikzcd}
\Sigma^{s,c+k2^{\gamma(n)}}S_{C_2}\ar[r]&\Sigma^{s,c+k2^{\gamma(n)}}\Cof(\rho^{n+1})\ar[r,"\tau^{k2^{\gamma(n)}}","\simeq"']&\Sigma^{s,c}\Cof(\rho^{n+1})\ar[r,"\partial"]&\Sigma^{s-n,c+1}S_{C_2}
\end{tikzcd},\end{center}
regarded as an element of $\pi_{n,k2^{\gamma(n)}-1}S_{C_2}$.
\tqed
\end{construction}

The following diagram may help illustrate the above definition:

\begin{center}\begin{tikzcd}[column sep=large]
\Sigma^{s-n-1,c}S_{C_2}\ar[d,"\rho^{n+1}"]\\
\Sigma^{s,c}S_{C_2}\ar[r,"\alpha"]\ar[d]&X\\
\Sigma^{s,c}\Cof(\rho^{n+1})\ar[ur,"\alpha'"]\ar[d,"\partial"]&\Sigma^{s,c+k2^{\gamma(n)}}\Cof(\rho^{n+1})\ar[l,"\tau^{k2^{\gamma(n)}}"']&\Sigma^{s,c+k2^{\gamma(n)}}S_{C_2}\ar[ul,"\tau^{k2^{\gamma(n)}}\alpha"']\ar[l]\ar[dll,"\text{indeterminacy}"]\\
\Sigma^{s-n,c+1}S_{C_2}
\end{tikzcd}.\end{center}

We will mostly make use of $\tau^{4k}$ as defined on $\rho^3$-torsion elements, which has indeterminacy contained in the image of multiplication by some element of $\pi_{2,4k-1}S_{C_2}$.

\begin{rmk}
We do not need the precise value of the indeterminacy of $\tau^4$ as defined on $\rho^3$-torsion elements, but expect that one may arrange for it to be equal to $\rho (\nu_{\cl}-\rho^4\sigma_{C_2})$, where $\nu_{\cl}\in \pi_{3,3}S_{C_2}$ is the nonequivariant quaternionic Hopf map and $\sigma_{C_2}\in\pi_{7,3}S_{C_2}$ is the $C_2$-equivariant octonionic Hopf map, up to choices of orientations; this element is detected by $\rho\tau^2 h_2$ in the $\bbR$-motivic summand of the $E_2$-page of the $C_2$-equivariant Adams spectral sequence. Assuming this is the case, one may interpret the construction as saying that $\tau^{4}\alpha \in \langle \rho(\nu_{\cl}-\rho^4\sigma_{C_2}),\rho^3,\alpha\rangle$.
\tqed
\end{rmk}

At this point, we resume our convention that everything is $K(1)$-local.

\section{The \texorpdfstring{$K(1)$}{K(1)}-local sphere}\label{sec:sphere}

At last we have everything needed for the main computation. Throughout this section, we abbreviate
\[
\pi_{\ast,\ast} = \pi_{\ast,\ast}b(S).
\]

We also fix throughout this section some $k\in\bbZ_2^\times$ which projects to a generator of the pro-cyclic group $\bbZ_2^\times/\{\pm 1\}$. The standard choice is $k = 3$.

This section should be regarded as running parallel to \cref{ssec:tables}, which organizes the output of our computation into tables, rather than preceding it.

\subsection{Additive structure}\label{ssec:sgroups}

The unit map $b(S)\rightarrow b(KO)$ factors through an equivalence
\[
b(S)\simeq b(KO^{\h \bbZ\{\psi^k\}})\simeq b(KO)^{\h \bbZ\{\psi^k\}}.
\]
This is realized by a fiber sequence
\begin{center}\begin{tikzcd}
b(S)\ar[r]&b(KO)\ar[r,"\psi^k-1"]&b(KO)
\end{tikzcd},\end{center}
which gives rise to short exact sequences
\begin{align*}
0\rightarrow\coker(\psi^k-1\colon \pi_{s+1,c+1}b(KO))\rightarrow\pi_{s,c}\rightarrow \ker(\psi^k-1\colon \pi_{s,c}b(KO))\rightarrow 0.
\end{align*}
We may view these short exact sequences as describing the extension problem for the HFPSS
\[
H^\ast(\bbZ\{\psi^k\};\pi_{\ast,\ast}b(KO))\Rightarrow\pi_{\ast,\ast};
\]
the kernel is the $0$-line $H^0$, the cokernel is the $1$-line $H^1$, and all lines above these are zero. These groups may be readily computed using our knowledge of $\pi_{\ast,\ast}b(KO)$.

We first fix some notation. Given $x\in \pi_{s,c}b(KO)$, write $\ol{x}$ for the image of $x$ in the $1$-line, detecting a unique element of $\pi_{s-1,c-1}$. Given $x\in \pi_{s,c}b(KO)$ fixed by $\psi^k$, write $x$ for the corresponding element of the $0$-line, detecting an element of $\pi_{s,c}$.

Recall the definition of $j_a$ from \cref{ssec:classic}, as well as the notation $S_{4a-1} = \bbZ_2/(2^{j_a})$. Given integers $a$ and $b$, define the $2$-adic unit
\[
u_{a,b} = 2^{j_a-j_{b-a}}\frac{k^{2b}-k^{2a}}{k^{2a}-1},
\]
taking $u_{a,b} = 1$ if $a = 0$ or $a = b$. Given symbols $x$ and $y$, define the group
\begin{align*}
E_{a,b}\{x,y\} &= \bbZ_2\{x,y\}/\left((k^{2a}-1)x+\tfrac{1}{2}(k^{2b}-k^{2a})y,~~(k^{2b}-1)y\right)\\
&\cong \bbZ_2\{x,y\}/(2^{j_a}x + 2^{j_{b-a}-1}u_{a,b}y,2^{j_b}y).
\end{align*}
This sits in a short exact sequence
\[
0\rightarrow S_{4b-1}\{y\}\rightarrow E_{a,b}\{x,y\}\rightarrow S_{4a-1}\{x\}\rightarrow 0,
\]
except when $a = 0$ where $E_{0,b}\{x,y\} = \bbZ_2\{x\} \oplus\bbZ_2/(2^{j_b-1})\{y\}$.

\begin{lemma}\label{lem:hfpss}
The full additive structure of the HFPSS is as described by the first, third, and fourth columns of \cref{table:basis}.
\end{lemma}
\begin{proof}
The groups $\pi_{\ast,\ast}b(KO)$ and their action by $\psi^k$ are as described in \cref{ssec:ko}. The lemma now follows by a direct calculation of the kernel and cokernel of $\psi^k-1$.
\end{proof}

\begin{lemma}\label{lem:indet}
All $\tau^4$-periodicity in $\pi_{\ast,\ast}$ defined for $\rho^3$-torsion elements holds without indeterminacy in coweights not congruent to $-1$ mod $4$.
\end{lemma}
\begin{proof}
The secondary operation $\tau^{4b}$ as defined on $\rho^3$ elements has indeterminacy contained in the image of multiplication by some element of $\pi_{2,4b-1}$. By \cref{lem:hfpss}, this group is completely detected on the $1$-line, so products out of it can be computed algebraically in the HFPSS, with no danger of hidden extensions. By inspection, this group is killed by everything except various quantities of elements in coweights congruent to $0$ mod $4$, proving the lemma.
\end{proof}

\begin{lemma}\label{lem:noext}
There are no nontrivial additive extensions in the HFPSS.
\end{lemma}
\begin{proof}
There is only room for possible nontrivial additive extensions in degrees $(8a,4b+1)$, $(8a+1,4b+1)$, and $(8a+2,4b+1)$. These degrees consist of $\rho^3$-torsion elements, so by \cref{lem:indet} we may reduce to $b=2a$. This is a computation in the duals of the spectra $P_{-1}$, $P_0$, and $P_1$. By \cref{prop:pic}, this is just a computation in $S$. Here the lemma is known, as was recalled in \cref{ssec:classic}.
\end{proof}

\begin{prop}\label{prop:basis}
The full additive structure of $\pi_{\ast,\ast}$ is as described in \cref{table:basis}.
\end{prop}
\begin{proof}
This follows immediately from \cref{lem:hfpss} and \cref{lem:noext}.
\end{proof}

From here on we would like to use the new notation for elements of $\pi_{\ast,\ast}$ given in the second column of \cref{table:basis}, but some care is necessary to safely do so: if $x$ lies on the $0$-line of the HFPSS, then $x$ may not uniquely determine an element of $\pi_{\ast,\ast}$, but only a coset of the subgroup generated by elements detected on the $1$-line in the same degree. In order to completely determine $\pi_{\ast,\ast}$, we must be explicit about how we lift these classes to $\pi_{\ast,\ast}$.

Note that there is a unit map $S\rightarrow b(S)^{C_2}\simeq D(P_0)$, given by restriction along the projection $P_0\simeq\Sigma^\infty_+BC_2\rightarrow S$. This induces maps $S_a\rightarrow \pi_{a,a}$ making $\pi_{\ast,\ast}$ into an $S_\ast$-algebra.

\begin{defn}\label{def:zeroline}
The generators of $\pi_{\ast,\ast}$ detected on the $0$-line are chosen as follows, where $a$ and $b$ range through $\bbZ$:
\begin{enumerate}
\item The element $\tau^{2b}h\in \pi_{0,2b} = \pi_{2b}D(P_{-2b})$ is defined as $\tau^{2b}h = \tr(1)$, where $\tr\colon \pi_0 S\rightarrow \pi_{0,2b}$ is the transfer. Equivalently, at least up to a sign, $\tau^{2b}h$ is the image of $1$ under the boundary map associated to the cofibering $D(P_{-2b})\rightarrow D(P_{-2b-1})\rightarrow D(S^{-2b-1})$.
\item The element $\mu_{a,2a}\in\pi_{8a+1,8a+1}$ is defined as the image of $\mu_a\in S_{8a+1}$ under the unit map $S\rightarrow D(P_0)$, where $\mu_a\in S_{8a+1}$ was itself defined in \cref{ssec:classic}. In general, we define $\mu_{a,b} = \tau^{4b-8a}\mu_{a,2a}\in \pi_{8a+1,4b+1}$, which is unambiguous by \cref{lem:indet}.
\item The element $\omega_0\in\pi_{-1,0}$ is defined as the Hurewicz image of $\rho$.
\item In general, there are two elements in $\pi_{8a-1,0}$ detected by $\rho v^a$, and we may distinguish $\omega_a\in \pi_{8a-1,0}$ by asking that $\omega_a \cdot \mu_0 = \omega_0\cdot \mu_{a,0}$.
\item There are two elements in $\pi_{8a-1,0}$ detected by $\eta_0 v^a$, and \cref{lem:koforget} implies that these lift $\mu_a$ and $\mu_a+\mu_0^2\rho_a$ through the forgetful map $\pi_{8a+1,0}\rightarrow S_{8a+1}$. We choose $\eta_a\in\pi_{8a+1,0}$ to lift $\mu_a$.
\item The elements $\omega_0\omega_a$, $\omega_0^2\omega_a$, $\omega_0\omega_a$, $\omega_0\eta_a$, $\eta_0\eta_a$, $\eta_0^2\eta_a$, $\omega_0\mu_{a,b}$, $\eta_0\mu_{a,b}$, $\omega_0\eta_0\mu_{a,b}$, $\omega_0\mu_0\mu_{a,b}$, $\eta_0\mu_0\mu_{a,b}$, and $\omega_0\eta_0\mu_0\mu_{a,b}$ are defined as products as their notation suggests.
\tqed
\end{enumerate} 
\end{defn}

\begin{rmk}
We will generally abbreviate $\mu_0 = \mu_{0,0}$. This is the Hurewicz image of the nonequivariant Hopf map $\eta_{\cl}$, and in general $\mu_{a,b} = \tau^{4b-8a}\beta^{4a}\mu_0$.
\tqed
\end{rmk}

\begin{lemma}\label{lem:unit}
The unit map $S_\ast\rightarrow \pi_{\ast,\ast}$ is given on generators by
\[
\rho_a\mapsto\rho_{a,2a},\qquad \mu_a\mapsto\mu_{a,2a},\qquad\xi_a\mapsto\xi_{a,2a+1}.
\]
\end{lemma}
\begin{proof}
Immediate from the definitions.
\end{proof}

\begin{lemma}
The forgetful map $\pi_{\ast,\ast'}\rightarrow S_\ast$ is as described in the last column of \cref{table:basis}.
\end{lemma}
\begin{proof}
Most of these values follow directly by comparing the HFPSSs used to compute $\pi_{\ast,\ast'}$ and $S_\ast$, together with our good choice of generators for $\pi_{\ast,\ast}$ given in \cref{def:zeroline}. There are two spots where a filtration jump occurs. The first is $\eta_a\mapsto \mu_a$, which holds by our choice of $\eta_a$. The second is $\omega_a\mapsto 2^{j_{2a}-1}\rho_a$. This is clear if $a = 0$, so suppose $a \neq 0$. In this case, $\omega_a$ is not in the image of $\rho$, but $2\omega_a$ is. Thus $\omega_a$ is sent under the forgetful map to some nonzero simple $2$-torsion element in $S_{8a-1}$, and $2^{j_{2a}-1}\rho_a$ is the only possibility.
\end{proof}

\begin{rmk}
As $\eta_a\eta_b\eta_c$ is not divisible by $\rho$, it is not in the kernel of the forgetful map, implying that $\mu_a\mu_b\mu_c\neq 0$ in $S_\ast$. This recovers the classical hidden extension $\mu_a\mu_b\mu_c = 4 \xi_{a+b+c}$, as there are no other nonzero simple $2$-torsion elements in this degree.
\tqed
\end{rmk}

We end this subsection with a graphical depiction of the groups $\pi_{\ast,\ast}$: 
\begin{center}
$\pi_{\ast,\ast}b(S_{K(1)})$
\includegraphics[page=1,width=\textwidth]{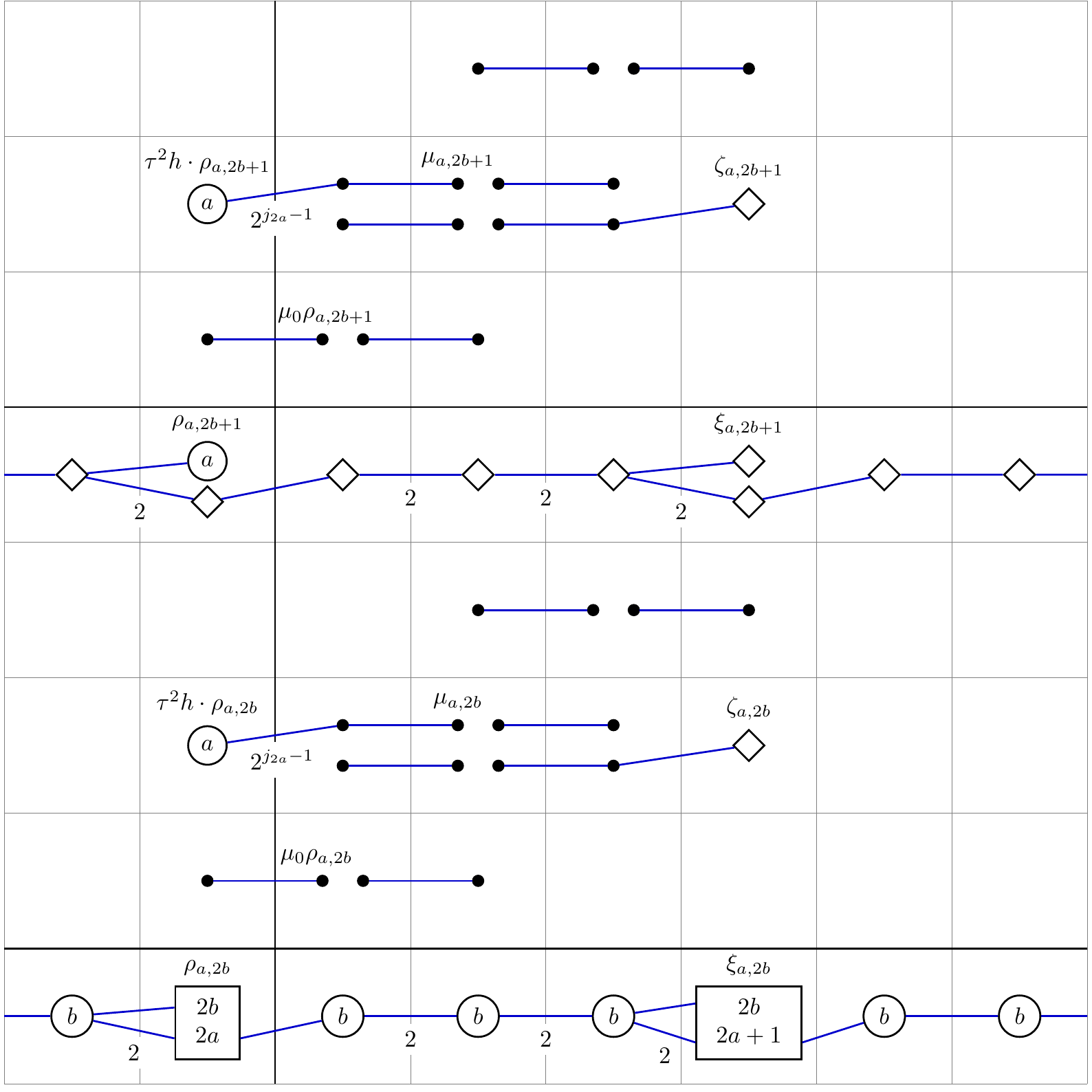}
\end{center}
\[
\vcenter{\hbox{\begin{tikzpicture}\draw(0,0) node[draw, fill, , circle, scale = 0.5] {};\end{tikzpicture}}} = \bbZ/(2); \qquad
\vcenter{\hbox{\begin{tikzpicture}\draw(0,0) node[draw, diamond, thick] {};\end{tikzpicture}}} = \bbZ/(8);\qquad
\vcenter{\hbox{\begin{tikzpicture}\draw(0,0) node[draw, circle, thick] {$a$};\end{tikzpicture}}} = S_{8a-1}; \qquad
\vcenter{\hbox{\begin{tikzpicture}\draw(0,0) node[draw, rectangle, thick] {\begin{tabular}{c}$b$\\$a$\end{tabular}};\end{tikzpicture}}} = E_{a,b};\qquad
\vcenter{\hbox{\begin{tikzpicture}\draw[{blue!80!black}, thick] (0,0) -- (1,0);\end{tikzpicture}}} = \rho.
\]
This depicts the $8\times 8$ region of the bigraded groups $\pi_{\ast,\ast}$, starting with $\pi_{8a-2,8b-1}$ in the bottom left corner. Observe for example how the classes along the main diagonal of slope $1$ form the same pattern seen in $S_\ast$, reflecting the equivalence $P_{8(a-b)-1}\simeq S$. This picture requires $a\neq 0$ and $b\neq 0$, as the $a = 0$ line additionally contains the elements $\tau^{2b}h$, and the $b = 0$ line contains the subring of $\pi_{\ast,0}$ generated by the elements $\omega_a$ and $\eta_a$ described in \cref{lem:root} below.

\subsection{Products}\label{ssec:products}

All products in $\pi_{\ast,\ast}$ involving at least one element detected on the $1$-line can be computed in the HFPSS, as there is no room for a hidden extension to occur. These products are summarized, together with all other products, in \cref{table:products}, and we will not go through them individually. However we do record the following explicitly, as it may not be apparent from the notation.

\begin{lemma}\label{lem:unhidden}
$\omega_a\cdot\zeta_{b,c} = \eta_0\mu_0^2\rho_{a+b,c}$ and $\eta_a\cdot\tau^{4b+2}h = \omega_0\mu_0\mu_{a,b}$.
\end{lemma}
\begin{proof}
These are inherited from \cref{lem:koext2} and \cref{lem:koext1}.
\end{proof}

On the other hand there is plenty of room for hidden products among elements detected on the $0$-line, and we resolve these here. Unfortunately, we know of no good way of visualizing this computation, as the the spectral sequence under consideration is concentrated on only two lines. However, we can determine where hidden extensions may lie with the aid of \cref{table:basis}, by manually checking all products of the form $x\cdot y$ where $x$ and $y$ are detected on the $0$-line and the degree of $x\cdot y$ contains one or more classes detected on the $1$-line. In doing so, we need only produce enough relations to reduce all such products to linear combinations of the given additive generators of $\pi_{\ast,\ast}$.

Throughout the following, $a,b,c$ range through $\bbZ$, and we write $\varphi\colon \pi_{\ast,\ast'}\rightarrow S_\ast$ for the forgetful map, which we recall has kernel equal to the image of multiplication by $\omega_0 = \rho$.

\begin{lemma}\label{lem:percomm}
$\omega_a\cdot \mu_{b,c} = \omega_0 \mu_{a+b,c}$ and $\mu_{a,b}\cdot\mu_{c,d} = \mu_0\cdot\mu_{a+c,b+d}$.
\end{lemma}
\begin{proof}
Consider the first product. As both $\tau^4$-periodicity and $\beta^4$-periodicity are defined without indeterminacy in these degrees, we may compute
\[
\omega_a\cdot\mu_{b,c} = \tau^{4c-8b}\beta^{4b}(\omega_a\cdot \mu_0) = \tau^{4c-8b}\beta^{4b}(\omega_0\cdot\mu_{a,0}) = \omega_0\cdot\mu_{a+b,c}.
\]
The second product is analogous.
\end{proof}

\begin{lemma}\label{lem:specialper}
$\eta_a\cdot  \mu_{b,c} = \eta_0\mu_{a+b,c}$.
\end{lemma}
\begin{proof}
It suffices to show that this identity holds mod $(2,\rho^3)$. By definition, $\eta_a$ is detected by $\eta_0\beta^{4a}\tau^{-8a}$. The further assumption that $\eta_a$ lifts $\mu_a = \beta^{4a}\mu_0$ ensures that $\eta_a = \tau^{-8a}\cdot\beta^{4a}\cdot \eta_0$ in $b(S/2)\otimes \Cof(\rho^3)$. Thus we may compute
\[
\eta_a\cdot\mu_{b,c} =\tau^{-8a}\cdot\beta^{4a}\cdot \eta_0\cdot\tau^{4c-8b}\cdot\beta^{4b}\cdot \mu_0 = \eta_0\cdot\mu_{a+b,c}
\]
as claimed.
\end{proof}

\begin{lemma}\label{lem:omegaeta2}
$\omega_0^2\eta_a = 2\omega_a$.
\end{lemma}
\begin{proof}
The alternative is that $\omega_0^2\eta_a = 2 \omega_a + \omega_0\mu_0\rho_{a,b}$, and thus $\mu_0\cdot\omega_0^2\eta_a = \omega_0\mu_0^2\rho_{a,b} \neq 0$. By \cref{lem:percomm}, $\mu_0\cdot\omega_0^2\eta_a = \eta_0\cdot\omega_0^2\mu_{a,0}$. The product $\omega_0^2\mu_{a,0}$ lives in $\pi_{8a-1,1} = S_{8a-1}$, and so must be divisible by $2$. As $\pi_{8a,1}$ is killed by $2$, it follows that $\mu_0\cdot\omega_0^2\eta_a = \eta_0\omega_0^2\mu_{a,0} = 0$.
\end{proof}

\begin{lemma}\label{lem:etaomega2}
$\eta_0^2\omega_a = 2 \eta_a$.
\end{lemma}
\begin{proof}
The alternative is that $\eta_0^2\omega_a = 2 \eta_a + \eta_0\mu_0\rho_{a,0}$. But $\varphi(\eta_0^2\omega_a) = \mu_0^2\cdot 2^{j_{2a}-1}\rho_a = 0$, whereas $\varphi(2\eta_a + \eta_0\mu_0\rho_{a,0}) = 2\mu_a + \mu_0^2\rho_a = \mu_0^2\rho_a\neq 0$.
\end{proof}

\begin{lemma}\label{lem:omegaeta}
$\omega_a\eta_{b+c} = \omega_{a+b}\eta_c$.
\end{lemma}
\begin{proof}
This identity holds on the $0$-line, so as $\mu_0$ acts injectively on the $1$-line in this degree, it suffices to show that the identity holds after multiplication by $\mu_0$. Here we have
\[
\mu_0\cdot\omega_a\eta_{b+c} = \mu_{a+b+c,0}\cdot\omega_0\eta_0 = \mu_0\cdot\omega_{a+b}\eta_c
\]
by \cref{lem:percomm} and \cref{lem:specialper}.
\end{proof}

It is worth combining the preceding three lemmas into the following.

\begin{prop}\label{lem:root}
Let $R\subset\pi_{\ast,0}$ be spanned by the elements $1$, $\omega_0^3\omega_a$, $\omega_0^2\omega_a$, $\omega_0\omega_a$, $\omega_a$, $\omega_0\eta_a$, $\eta_a$, $\eta_0\eta_a$, and $\eta_0^2\eta_a$. Then $R$ is a subring, and satisfies the following relations:
\begin{gather*}
\omega_0^2\eta_a = 2 \omega_a,\qquad \eta_0^2\omega_a = 2 \eta_a,\qquad \eta_0^3\eta_a = \omega_0^3\omega_{a+1} \\
\eta_{a+b}\eta_c = \eta_a\eta_{b+c},\qquad\eta_{a+b}\omega_c = \eta_a\omega_{b+c},\qquad\omega_{a+b}\omega_c = \omega_a\omega_{b+c}.
\end{gather*}
\end{prop}
\begin{proof}
The first, second, and fourth relations were proved in \cref{lem:omegaeta2}, \cref{lem:etaomega2}, and \cref{lem:omegaeta} respectively. The remaining relations hold as they hold on the $0$-line and there are no classes detected on the $1$-line in their respective degrees.
\end{proof}

\begin{rmk}
As $\omega_0 = \rho$, there holds in $R$ the relation
\[
\rho\cdot 2^{4c+\epsilon} = \begin{cases}
\rho^{8c+2}\cdot\eta_c&\epsilon = 1,\\
\rho^{8c+3}\cdot\eta_0\eta_c&\epsilon = 2,\\
\rho^{8c+4}\cdot\eta_0^2\eta_c&\epsilon = 3,\\
\rho^{8c+8}\cdot\omega_{c+1}&\epsilon=4.
\end{cases}
\]
This is a manifestation of the Mahowald invariants of powers of $2$ \cite{mahowaldravenel1993root}, as seen through the $C_2$-equivariant interpretation of the Mahowald invariant \cite{brunergreenlees1995bredon}. In particular, we expect that the classes $\omega_a$ and $\eta_a$ lift to the non-localized $C_2$-equivariant stable stems $\pi_{\ast,0}S_{C_2}$ for $a\geq 0$, and the subring of $R$ generated by these is exactly the Bredon--Landweber region of $\pi_{\ast,\ast}S_{C_2}$, as has also been studied by Guillou--Isaksen \cite{guillouisaksen2019bredonlandweber} via the $C_2$-equivariant Adams spectral sequence.
\tqed
\end{rmk}

\begin{lemma}\label{lem:rooth1}
$\omega_a\cdot\tau^{2b}h = 0$.
\end{lemma}
\begin{proof}
There is no room for a hidden extension when $b$ is odd, so suppose that $b$ is even. The alternative is that $\omega_a\tau^{2b}h = \omega_0\mu_0^2\rho_{a,b/2}$, and thus $\eta_0\omega_a\tau^{2b}h = \omega_0\eta_0\mu_0^2\rho_{a,b/2}\neq 0$. But $\eta_0\omega_a\tau^{2b}h = \eta_a\omega_0\tau^{2b}h$, so we reduce to verifying $\omega_0\tau^{2b}h = 0$. This now follows from the definition of $\tau^{2b}h$.
\end{proof}

\begin{lemma}\label{lem:rooth2}
$\eta_a\cdot\tau^{4b}h = 0$.
\end{lemma}
\begin{proof}
The alternative is that $\eta_a\cdot\tau^{4b}h = \eta_0\mu_0\rho_{a,b}$, but this cannot hold as $\omega_0\cdot \eta_a\cdot\tau^{4b}h = 0$ whereas $\omega_0\cdot\eta_0\mu_0\rho_{a,b}\neq 0$.
\end{proof}

\begin{lemma}\label{lem:hh} 
$\tau^{2a}h\cdot\tau^{2b}h = 2 \tau^{2(a+b)}h.$
\end{lemma}
\begin{proof}
The alternative is that $\tau^{2a}h\cdot\tau^{2b}h = 2 \tau^{2(a+b)}h + \mu_0\rho_{0,(a+b)/2}$ for some $a,b$ such that $a+b$ is even, but this cannot hold as $\omega_0\cdot\tau^{2a}h\cdot \tau^{2b}h = 0$ whereas $\omega_0\cdot\mu_0\rho_{0,(a+b)/2} \neq 0$.
\end{proof}

\begin{lemma}\label{lem:omegamu}
$\omega_a\cdot\omega_0\mu_{b,c} = 2^{j_{2(a+b)}-1}\tau^2h\cdot\rho_{a+b,c}.$
\end{lemma}
\begin{proof}
As
$
\omega_a \omega_0\mu_{b,c} = \omega_0\omega_0\mu_{a+b,c}
$
by \cref{lem:percomm}, we may reduce to the case $a=0$. As $\varphi(\tau^2h\cdot\rho_{c,b})=2\rho_c$, we have 
$
\varphi(2^{j_{2c}-1}\tau^2h\cdot\rho_{c,b}) = 0.
$
Thus $2^{j_{2c}-1}\tau^2h\cdot\rho_{c,b}$ is in the image of $\omega_0$, and this is the only possibility.
\end{proof}

\begin{samepage}
\begin{lemma}\label{lem:etamu}
$\eta_a\cdot\eta_0\mu_{b,c} = 4 \zeta_{a+b,c}.$
\end{lemma}
\begin{proof}
This is the only possibility given $\varphi(\eta_a\cdot\eta_0\mu_{b,c}) = \mu_a\mu_0\mu_b = 4\xi_{a+b} = \varphi(4\zeta_{a+b,c})$.
\end{proof}
\end{samepage}

\begin{samepage}
\begin{lemma}\label{lem:mucubed2}
$\mu_0^2\cdot \mu_{a,2b} = 4 \xi_{a,2b+1}.$
\end{lemma}
\begin{proof}
Here we have
\[
\mu_0^2\mu_{a,2b}\in E_{2a+1,2b+1}\{\xi_{a,2b+1},\omega_0\eta_0\xi_{a,2b+1}\} \cong \bbZ/(8)\{\xi_{a,2b+1}\}\oplus \bbZ/(8)\{\omega_0\eta_0\xi_{a,2b+1}\},
\]
so by comparison with $\varphi(\mu_0^2\mu_{a,2b}) = \mu_0^2\mu_a = 4\xi_a$ we find that 
\[
\mu_0^2\mu_{a,2b} \in 4 \xi_{a,2b+1} + \bbZ/(2)\{4\omega_0\eta_0\xi_{a,2b+1}\}.
\]
The elements $\xi_{a,2b+1}$ are $\rho^5$-torsion, and there is a $\tau^8$-periodicity of $\xi_{a,2(b+c)+1}\in\tau^{8c}\xi_{a,2b+1}$. This has indeterminacy contained in the image of multiplication by an element of $\pi_{4,8c-1} = S_{8c-1}\{\eta_0\xi_{0,2c}\}$, which is zero in this degree. We may thus reduce to $b=a$, where we are claiming $\mu_0^2\mu_{a,2a} = 4 \xi_{a,2a+1}$. This is the image of the relation $\mu_0^2\mu_a = 4 \xi_a$ recalled in \cref{lem:sgroups} under the unit map $S_\ast\rightarrow\pi_{\ast,\ast}$ given in \cref{lem:unit}, proving the lemma.
\end{proof}
\end{samepage}

\begin{lemma}\label{lem:mucubed1}
$\mu_0^2\cdot\mu_{a,2b-1} = 4\xi_{a,2b}+2 u_{2a+1,2b} \omega_0\eta_0\xi_{a,2b}.$
\end{lemma}
\begin{proof}
Abbreviate 
\[
x = (2\sqrt{\per})\tau^{8b}\per^a,\qquad y = \rho\eta_0(2\sqrt{\per})\tau^{8b}\per^a,
\]
so $x,y\in \pi_{8a+4,8b}b(KO)$ with $\xi_{a,2b} = \ol{x}$ and $\omega_0\eta_0\xi_{a,2b} = \ol{y}$, and also abbreviate $u = u_{2a+1,2b}$. The product under consideration is a simple $2$-torsion element of the group
\[
\pi_{8a+3,8b-1} = E_{2a+1,2b}\{\ol{x},\ol{y}\} = \bbZ_2\{\ol{x},\ol{y}\}/(8\ol{x}+4u\ol{y},2^{j_{2b}}\ol{y}),
\]
so by comparison with $\varphi(\mu_0^2\mu_{a,2b-1}) = \mu_0^2 \mu_a = 4\xi_a$, we find
\[
\mu_0^2\mu_{a,2b-1} \in 4 \ol{x} +2u\ol{y} + \bbZ/(2)\{2^{j_{2b}-1}\ol{y}\}.
\]
This leaves two possible values, and we must cut it down to one.

By definition, $\pi_{8a+3,8b-1} = \pi_{8b-1}D(P_{8(a-b)+4})$, and by \cref{prop:cell} we may identify $D(P_{8(a-b)+4})\simeq \Cof(4\rho_{b-a-1/2})$. Let $X = \Cof(2\rho_{b-a-1/2})$. This comes equipped with a map $p\colon X\rightarrow D(P_{8(a-b)+4})$ of degree $2$ on the $0$-cell, and $\mu_{a,2b-1}$ lifts to $\pi_{8b-3}X$. The product under consideration then lifts to a simple $2$-torsion class in $\pi_{8b-1}X$, and we claim that the indicated relation is the only possibility.

We can choose $KO_{8b}X = \bbZ_2\{x',y'\}$ in such a way that $p(x')=x$ and $p(y') = 2y$, so the action of $\psi^k$ is given by
\begin{align*}
\psi^k(x') &= k^{4a+2}x' + \tfrac{1}{4}(k^{4b}-k^{4a+2})y'\\
 \psi^k(y') &= k^{4b}y'.
\end{align*}
It follows that
\[
\pi_{8b-1} X = \bbZ_2\{\ol{x'},\ol{y'}\}/(8x' + 2 u \ol{y'}, 2^{j_{2b}}\ol{y'}),
\]
with $p(\ol{x'}) = \ol{x}$ and $p(\ol{y'}) = 2\ol{y}$. The class $4\ol{x}+(2u+2^{j_{2b}-1})\ol{y}$ does not lift to a simple $2$-torsion element of $\pi_{8b-1}X$, so the situation is as claimed.
\end{proof}

It is convenient to record the following consequence of the preceding two lemmas.

\begin{lemma}\label{lem:omegamumu}
$\omega_0\mu_0^2\mu_{a,b} = 2^{j_{b+1}-1}\eta_0^3\rho_{a,b+1}$.
\end{lemma}
\begin{proof}
First suppose that $b = 2c$ is even. In this case \cref{lem:mucubed2} implies that
\[
\omega_0\cdot\mu_0^2\mu_{a,2c} = \omega_0\cdot 4 \xi_{a,2c+1} = 4\eta_0^3\rho_{a,2c+1}.
\]
Next suppose that $b = 2c-1$ is odd, and define the $2$-adic unit $u = 8\cdot (k^{4a+2}-1)^{-1}$. Now \cref{lem:mucubed1} implies that
\begin{align*}
\omega_0\cdot\mu_0^2\mu_{a,2c-1} &= \omega_0\cdot(4\xi_{a,2c}+2u_{2a+1,2c}\omega_0\eta_0\xi_{a,2c}) \\
&= u\omega_0\cdot (\tfrac{1}{2}(k^{4a+2}-1)\xi_{a,2c} + \tfrac{1}{4}(k^{4c}-k^{4a+2})\omega_0\eta_0\xi_{a,2c})\\
&= \tfrac{u}{2}(k^{4c}-1)\omega_0\xi_{a,2c} = 2^{j_{2c}-1} \eta_0^3\rho_{a,2c}
\end{align*}
as claimed.
\end{proof}

\begin{rmk}
\cref{lem:omegamumu} may also be derived as a consequence of \cref{lem:tr1} below. Indeed, \cref{lem:tr1} implies that $\tr_{4b+3}(\mu_0^2\rho_a) = 0$ and $\tr_{4b+3}(\mu_0\mu_a) = \omega_0\mu_0^2\mu_{a,b}$. As the image of the transfer equals the kernel of $\rho = \omega_0$, and $2^{j_{b+1}-1}\eta_0^3\rho_{a,b+1}$ is in the kernel of $\omega_0$ by \cref{lem:root}, the only possibility is that $\omega_0\mu_0^2\mu_{a,b} = 2^{j_{b+1}-1}\eta_0^3\rho_{a,b+1}$.
\tqed
\end{rmk}

\begin{samepage}
\begin{lemma}\label{lem:muh1}
$\mu_{a,b}\cdot\tau^{4c}h = \omega_0\eta_0\mu_{a,b+c}.$
\end{lemma}
\begin{proof}
By $\tau^4$-periodicity, we may reduce to $c=0$. Here we compute
\[
\mu_{a,b}\cdot h = \mu_{a,b}\cdot (2 - \omega_0\eta_0) = \omega_0\eta_0 \mu_{a,b}
\]
as claimed.
\end{proof}
\end{samepage}

\begin{lemma}\label{lem:muh2}
$\mu_{a,b}\cdot\tau^{4c-2}h = 2^{j_{b+c}-1} \eta_0^2\rho_{a,b+c}.$
\end{lemma}
\begin{proof}
As $\eta_0$ acts injectively in this degree, it is sufficient to show
\[
\eta_0\cdot\mu_{a,b}\cdot\tau^{4c-2}h = 2^{j_{b+c}-1} \eta_0^3\rho_{a,b+c}.
\]
This lives in the group $S_{4(b+c)-1}=\bbZ_2/(2^{j_{b+c}})$, so it is sufficient to show that $\eta_0\cdot\mu_{a,b}\cdot\tau^{4c-2}h$ is nonzero except when $b+c=0$. Using \cref{lem:unhidden} and \cref{lem:percomm}, we may compute
\[
\eta_0\cdot\mu_{a,b}\cdot\tau^{4c-2}h = \mu_{a,b}\cdot \omega_0\mu_0\cdot\mu_{0,c-1} = \omega_0\cdot\mu_0^2\cdot\mu_{a,b+c-1}.
\]
This is nonzero for $b+c\neq 0$ by \cref{lem:omegamumu}.
\end{proof}

\subsection{Transfers}\label{ssec:transfers}

We now consider the transfer maps
\[
\tr_c\colon S_\ast\rightarrow \pi_{\ast,c}.
\]
We first compute the values $\tr_c(1) \in \pi_{0,c}$. For the following, recall that the image of the transfer equals the kernel of $\rho = \omega_0$.

\begin{lemma}
$\tr_{2b}(1) = \tau^{2b}h$.
\end{lemma}
\begin{proof}
This holds by \cref{def:zeroline}.
\end{proof}

\begin{lemma}
$\tr_{4b-1}(1) = 2^{j_b-1}\eta_0\rho_{0,b}$.
\end{lemma}
\begin{proof}
This is implicit in \cref{prop:cell}, but let us proceed directly. The transfer $\tr_{4b-1}(1)$ lives in the group $\pi_{0,4b-1} = \bbZ/(2^{j_b})\{\eta_0\rho_{0,b}\}$. As 
\[
\pi_{-1,4b-1} \cong E_{0,b}\{\rho_{0,b},\omega_0\eta_0\rho_{0,b}\} \cong \bbZ_2\{\rho_{0,b}\}\oplus\bbZ_2/(2^{j_b-1})\{\omega_0\eta_0\rho_{0,b}\},
\]
the kernel of $\omega_0$ acting on $\pi_{0,4b-1}$ is generated by $2^{j_b-1}\eta_0\rho_{0,b}$, and the stated value of $\tr_{4b-1}(1)$ is the only possibility.
\end{proof}

\begin{samepage}
\begin{lemma}\label{lem:tr1}
$\tr_{4b+1}(1) = \omega_0\mu_{0,b}$.
\end{lemma}
\begin{proof}
The transfer is $\tau^1$-periodic, and the indeterminacy of $\tau^1$ is contained in the image of multiplication by $h = \tr_0(1)$. In particular, as
\[
\tr_{4b+1}(1)\in\bbZ/(2)\{\omega_0\mu_{0,b},\omega_0\mu_0^2\rho_{0,b}\},
\]
the indeterminacy vanishes in these degrees, so we may reduce to the case $b = 0$. Consider the $2$-primary but otherwise unlocalized $C_2$-equivariant stable stems $\pi_{\ast,\ast}S_{C_2}$. We have
\[
\pi_{0,1}S_{C_2} = \bbZ/(2)\{\rho\eta_{\cl},\rho^3\nu_{C_2}\},
\]
where $\nu_{C_2}$ is the $C_2$-equivariant quaternionic Hopf fibration, and there is a relation
\[
\rho^2\eta_{\cl} = \rho^4 \nu_{C_2}
\]
reflecting the Mahowald invariant $R(\eta_\cl) = \nu_\cl$. These facts may be read off the $\bbR$-motivic computations \cite{belmontisaksen2020rmotivic}, which agree with the corresponding $C_2$-equivariant computations in this range \cite{belmontguillouisaksen2021c2}; note that in their work $\rho(\eta_{\cl} + \rho^2\nu_{C_2})$ is detected by  $\rho\tau h_1$. In particular, as the image of the transfer equals the kernel of $\rho$, it must be that
\[
\tr_1(1) = \rho(\eta_{\cl} + \rho^2\nu_{C_2}) \in \pi_{0,1}S_{C_2}.
\]
As $S_{K(1)}$ detects $\nu_\cl$, necessarily $b(S_{K(1)})$ detects $\nu_{C_2}$, and the only possibility is that $\nu_{C_2}$ is detected by an odd multiple of $\zeta_{0,0}$. As $\omega_0^3 \zeta_{0,0} = 0$, it follows that $\tr_1(1) = \omega_0\mu_0\in\pi_{0,1}$.
\end{proof}
\end{samepage}

In general, the transfer is $S_\ast$-linear in the sense that
\[
\tr_c(x\cdot y) = x \cdot \tr_{c-|x|}(y)
\]
for $x,y\in S_\ast$. The product on the right makes use of the $S_\ast$-module structure on $\pi_{\ast,\ast}$ determined by \cref{lem:unit} and the ring structure of $\pi_{\ast,\ast}$. Combining the values of $\tr_c(1)$ computed above with the structure of the ring $\pi_{\ast,\ast}$ computed in \cref{ssec:products}, one may easily compute the values of $\tr_c(\alpha)$ for all $\alpha\in S_\ast$ and $c\in\bbZ$, and we collect these in \cref{table:transfers}. The following is a depiction of the resulting Mackey functor coefficients of $b(S)$:

\begin{center}
$\ul{\pi}_{\ast,\ast}b(S_{K(1)})$
\includegraphics[page=1,width=\textwidth]{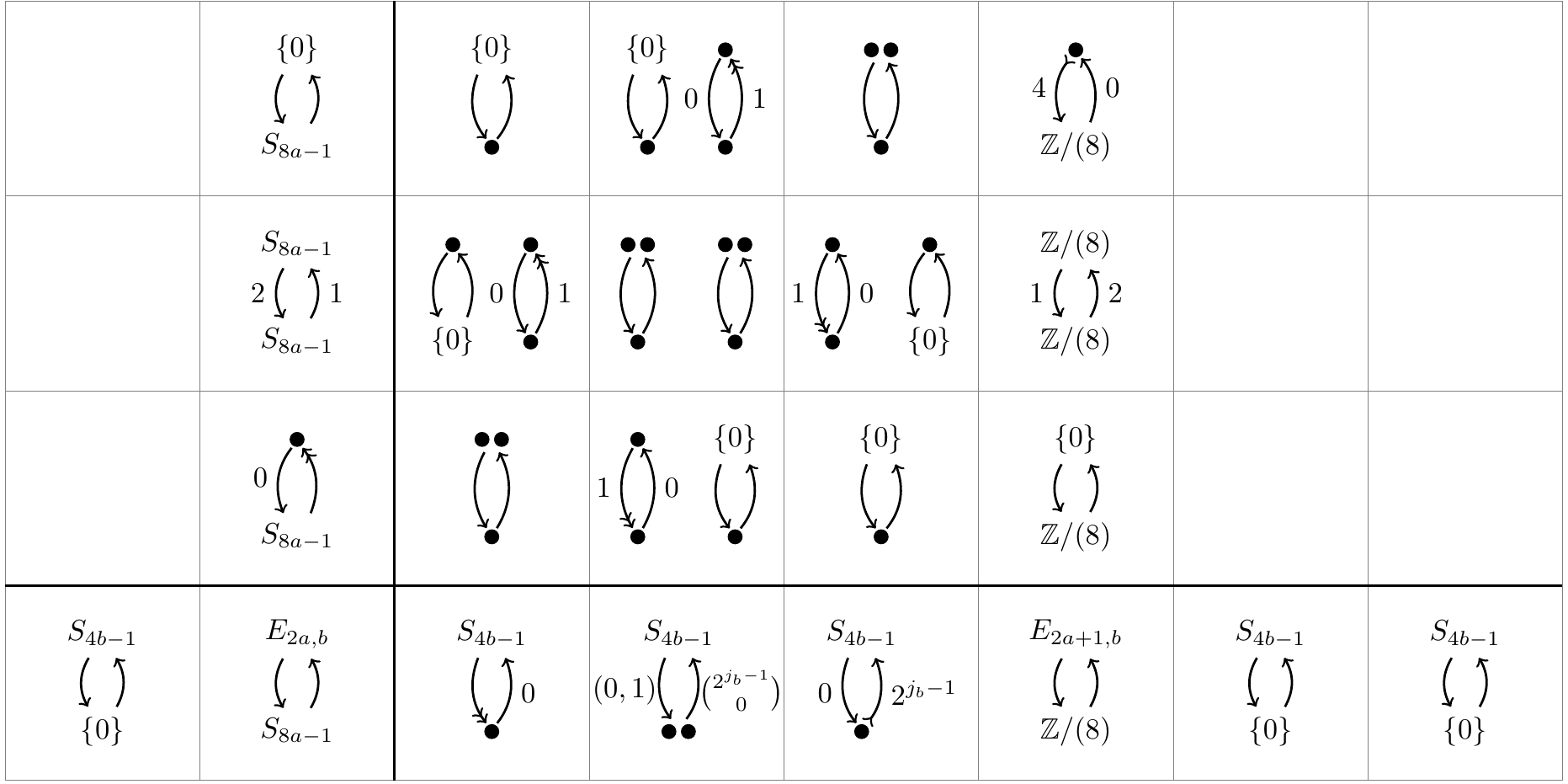}
\end{center}
\begin{gather*}
\vcenter{\hbox{\begin{tikzpicture}\draw(0,0) node[draw, fill, , circle, scale = 0.5] {};\end{tikzpicture}}} = \bbZ/(2); 
\qquad
\vcenter{\hbox{\begin{tikzpicture}[scale = 1.25]
\draw(-0.1,0.5) node[draw, fill, circle, scale = 0.35] {};
\draw(0,0.5) node (B1) {};
\draw(0.1,0.5) node[draw, fill, circle, scale = 0.35] {};
\draw(0,-0.5) node[draw, fill, circle, scale = 0.35] (B1b) {};
\draw[bend right,->,shorten >= 1pt, shorten <= 1pt] (B1) to node [left] {} (B1b);
\draw[bend right,->,shorten >= 1pt, shorten <= 1pt] (B1b) to node [right] {} (B1);\end{tikzpicture}}} =~~ \begin{tikzcd}\bbZ/(2)\times\bbZ/(2)\ar[d,bend right, "{(x,y)\mapsto x}"']\\\bbZ/(2)\ar[u,bend right, "{1\mapsto (0,1)}"']\end{tikzcd};\qquad
\begin{tikzcd}
E_{a,b}\ar[d, bend right]\\
S_{4a-1}\ar[u, bend right]\end{tikzcd} =~~ \begin{tikzcd}
E_{a,b}\{x,y\}\ar[d, bend right,"\substack{x\mapsto z \\ y \mapsto 0}"']\\
S_{4a-1}\{z\}\ar[u, bend right,"z\mapsto 2x-y"']\end{tikzcd}.
\end{gather*}
This depicts the $8\times 4$ region of the bigraded Mackey functor $\ul{\pi}_{\ast,\ast}$ starting with $\ul{\pi}_{8a-2,4b-1}$ in the bottom left corner, with each box corresponding to a single $\ul{\pi}_{s,c}$. This picture requires $a\neq 0$ and $b\neq 0$.  In general, the box corresponding to $\ul{\pi}_{0,0}$ or $\ul{\pi}_{0,2b}$ for $b\neq 0$ contains an additional term of the form 
\begin{center}
\begin{tikzcd}
\bbZ_2\oplus\bbZ_2\ar[d,bend right, "{(x,y)\mapsto x}"']\\
\bbZ_2\ar[u,bend right, "{1\mapsto (2,-1)}"']
\end{tikzcd}$\qquad$or$\qquad$
\begin{tikzcd}
\bbZ_2\ar[d,bend right,"2"']\\
\bbZ_2\ar[u,bend right, "1"']
\end{tikzcd}
\end{center}
respectively, and away from $\ul{\pi}_{0,0}$ the $\ul{\pi}_{\ast,0}$ line instead looks like the following:

\begin{center}
$\ul{\pi}_{\ast,0}b(S_{K(1)})$
\includegraphics[page=1,width=\textwidth]{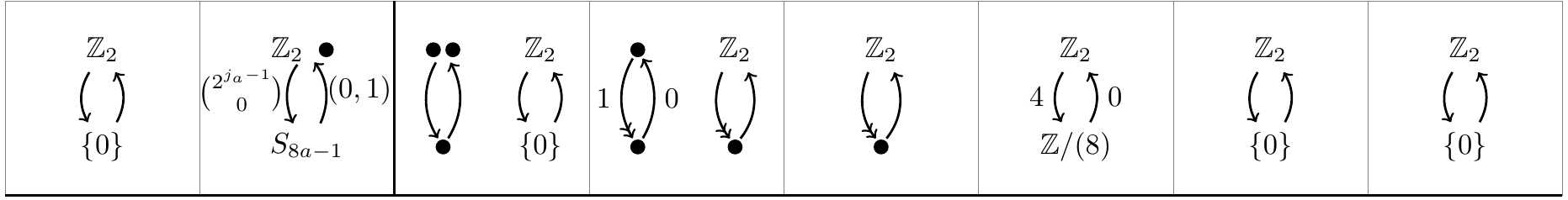}
\end{center}

These diagrams describe the Mackey functors $\ul{\pi}_{s,c}$ through their Lewis diagram \cite{lewis1988rog}
\begin{center}\begin{tikzcd}
\pi_{s,c}\ar[d,bend right,"\res"' pos=0.4]\\
\pi_s\ar[u,bend right, "\tr_c"' pos=0.6]
\end{tikzcd},\end{center}
where if multiple such diagrams appear in a single box then $\ul{\pi}_{s,c}$ is a direct sum of the corresponding Mackey functors. Our diagrams do not depict the action of $C_2$ on $\pi_s$, but the double coset formula implies that this action is given by just $\alpha\mapsto\res(\tr_c(\alpha)) - \alpha$.

This concludes the work necessary for \cref{thm:main}.

\section{Tables}
\label{ssec:tables}

The following tables summarize the structure of $\ul{\pi}_{\ast,\ast}b(S_{K(1)}) = \ul{\pi}_\star L_{KU_{C_2}/(2)}S_{C_2}$:

\begin{enumerate}\setcounter{enumi}{-1}
\item Recall $\ul{\pi}_{s,c}b(S_{K(1)}) = \ul{\pi}_{c+(s-c)\sigma}L_{KU_{C_2}/(2)}S_{C_2}$.
\item \cref{table:basis} gives a minimal set of additive generators. See \cref{ssec:sgroups}.
\item \cref{table:products} gives a full set of relations. See \cref{ssec:products}.
\item \cref{table:mult} gives a minimal set of multiplicative generators.
\item \cref{table:transfers} describes the transfers for $\ul{\pi}_{\ast,\ast}b(S_{K(1)})$. See \cref{ssec:transfers}.
\end{enumerate}
Recall also that $j_a$ is defined so that $\bbZ_2/(2^{j_a}) = \bbZ_2/(8a) = \pi_{4a-1}S_{K(1)} = S_{4a-1}$, and that we have defined
\[
u_{a,b} = 2^{j_a-j_{b-a}}\frac{k^{2b}-k^{2a}}{k^{2a}-1},\qquad E_{a,b} = \bbZ_2\{x,y\}/(2^{j_a}x + 2^{j_{b-a}-1}u_{a,b}y,2^{j_b}y),
\]
taking $u_{a,b} = 1$ if $a = 0$ or $a = b$, and where $k\in\bbZ_2^\times$ is a fixed element projecting to a generator of the pro-cyclic group $\bbZ_2^\times/\{\pm 1\}$, such as $k = 3$. Moreover, $\omega_0 = \rho$, $\mu_0 = \mu_{0,0}$, and $\eta_0 = -\eta_{C_2}$ (see \cref{rmk:eta}).

In all the following, $a,b,c,d$ range through $\bbZ$.

\begin{longtable}{lllll}
\caption{Additive generators} \\
\toprule
$(s,c)$ & Generator & Detected by & Group & Image in $\pi_\ast S_{K(1)}$ \\
\midrule \endhead
\bottomrule \endfoot
\label{table:basis}
$\!\!(0,0)$ & $1$ & $1$ & $\bbZ_2$ & $1$ \\
$(8a-4,0)$ & $\omega_0^3\omega_{a}$ & $\rho^4\per^a$,\quad\hspace{6pt} $\per = \beta^4\tau^{-8}$ & $\bbZ_2$ & $0$ \\
$(8a-3,0)$ & $\omega_0^2\omega_{a}$ & $\rho^3\per^a$ & $\bbZ_2$ & $0$ \\
$(8a-2,0)$ & $\omega_0\omega_{a}$ & $\rho^2\per^a$ & $\bbZ_2$ & $0$ \\
$(8a-1,0)$ & $\omega_{a}$ & $\rho\per^a$ & $\bbZ_2$ & $2^{j_{2a}-1}\rho_a$ \\
$(8a,0)$ & $\omega_0\eta_a$ & $\rho\eta_0\per^a$,\quad $\eta_0 = \rho\beta\tau^{-2}$ & $\bbZ_2$ & $0$ \\
$(8a+1,0)$ & $\eta_a$ & $\eta_0\per^a$ & $\bbZ_2$ & $\mu_a$ \\
$(8a+2,0)$ & $\eta_0\eta_a$ & $\eta_0^2\per^a$ & $\bbZ_2$ & $\mu_0\mu_a$ \\
$(8a+3,0)$ & $\eta_0^2\eta_a$ & $\eta_0^3\per^a$ & $\bbZ_2$ & $4\xi_a$ \\
\midrule
$(0,4b)$, $b\neq 0$ & $\tau^{4b}h$ & $\tau^{4b}h$,\quad\hspace{6pt} $h=2-\rho\eta_0$ & $\bbZ_2$ & $2$ \\
$(8a-1,4b)$ & $\omega_0\mu_0\rho_{a,b}$ & $\ol{\rho\mu\tau^{4b}\per^a}$ & $\bbZ/(2)$ & $0$ \\
$(8a,4b)$ & $\mu_0\rho_{a,b}$ & $\ol{\mu\tau^{4b}\per^a}$ & $\bbZ/(2)$ & $\mu_0\rho_a$ \\
$(8a,4b)$ & $\omega_0\eta_0\mu_0\rho_{a,b}$ & $\ol{\rho\eta_0\mu\tau^{4b}\per^a}$ & $\bbZ/(2)$ & $0$ \\
$(8a+1,4b)$ & $\eta_0\mu_0\rho_{a,b}$ & $\ol{\eta_0\mu\tau^{4b}\per^a}$ & $\bbZ/(2)$ & $\mu_0^2\rho_a$ \\
\midrule
$(8a-1,4b+1)$ & $\tau^2 h \cdot \rho_{a,b}$ & $\ol{\tau^{4b}(\tau^2 h)\per^a}$ & $S_{8a-1}$ & $2\rho_a$ \\
$(8a,4b+1)$ & $\omega_0\mu_{a,b}$ & $\rho\mu\tau^{4b}\per^a$ & $\bbZ/(2)$ & $0$ \\
$(8a,4b+1)$ & $\omega_0\mu_0^2\rho_{a,b}$ & $\ol{\rho\mu^2\tau^{4b}\per^a}$ & $\bbZ/(2)$ & $0$ \\
$(8a+1,4b+1)$ & $\mu_{a,b}$ & $\mu\tau^{4b}\per^a$ & $\bbZ/(2)$ & $\mu_a$ \\
$(8a+1,4b+1)$ & $\omega_0\eta_0\mu_{a,b}$ & $\rho\eta_0\mu\tau^{4b}\per^a$ & $\bbZ/(2)$ & $0$ \\
$(8a+1,4b+1)$ & $\mu_0^2\rho_{a,b}$ & $\ol{\mu^2\tau^{4b}\per^a}$ & $\bbZ/(2)$ & $\mu_0^2\rho_a$ \\
$(8a+1,4b+1)$ & $\omega_0\eta_0\mu_0^2\rho_{a,b}$ & $\ol{\rho\eta_0\mu^2\tau^{4b}\per^a}$ & $\bbZ/(2)$ & $0$ \\
$(8a+2,4b+1)$ & $\eta_0\mu_0^2\rho_{a,b}$ & $\ol{\eta_0\mu^2\tau^{4b}\per^a}$ & $\bbZ/(2)$ & $0$ \\
$(8a+2,4b+1)$ & $\eta_0\mu_{a,b}$ & $\eta_0\mu\tau^{4b}\per^a$ & $\bbZ/(2)$ & $\mu_0\mu_a$ \\
$(8a+3,4b+1)$ & $\zeta_{a,b}$ & $\ol{\tau^{4b}(\tau^2 h\sqrt{\per} )\per^a}$ & $\bbZ/(8)$ & $\xi_a$ \\
\midrule
$(0,4b+2)$ & $\tau^{4b+2}h$ & $\tau^{4b+2}h$ & $\bbZ_2$ & $2$ \\
$(8a+1,4b+2)$ & $\omega_0\mu_0\mu_{a,b}$ & $\rho\mu^2\tau^{4b}\per^a$ & $\bbZ/(2)$ & $0$ \\
$(8a+2,4b+2)$ & $\mu_0\mu_{a,b}$ & $\mu^2\tau^{4b}\per^a$ & $\bbZ/(2)$ & $\mu_0\mu_a$ \\
$(8a+2,4b+2)$ & $\omega_0\eta_0\mu_0\mu_{a,b}$ & $\rho\eta_0\mu^2\tau^{4b}\per^a$ & $\bbZ/(2)$ & $0$ \\
$(8a+3,4b+2)$ & $\eta_0\mu_0\mu_{a,b}$ & $\eta_0\mu^2\tau^{4b}\per^a$ & $\bbZ/(2)$ & $4\xi_a$ \\
\midrule
$(8a-4,4b-1)$ & $\omega_0^3\rho_{a,b}$ & $\ol{\rho^3\tau^{4b}\per^a}$ & $S_{4b-1}$ & $0$ \\
$(8a-3,4b-1)$ & $\omega_0^2\rho_{a,b}$ & $\ol{\rho^2\tau^{4b}\per^a}$ & $S_{4b-1}$ & $0$ \\
$(8a-2,4b-1)$ & $\omega_0\rho_{a,b}$ & $\ol{\rho\tau^{4b}\per^a}$ & $S_{4b-1}$ & $0$ \\
$(8a-1,4b-1)$ & $\rho_{a,b}$ & $\ol{\tau^{4b}\per^a}$ & $E_{2a,b}$ & $\rho_a$ \\
& $\omega_0\eta_0\rho_{a,b}$ & $\ol{\rho\eta_0\tau^{4b}\per^a}$ & Above & $0$ \\
$(8a,4b-1)$ & $\eta_0\rho_{a,b}$ & $\ol{\eta_0\tau^{4b}\per^a}$ & $S_{4b-1}$ & $\mu_0\rho_a$ \\
$(8a+1,4b-1)$ & $\eta_0^2\rho_{a,b}$ & $\ol{\eta_0^2\tau^{4b}\per^a}$ & $S_{4b-1}$ & $\mu_0^2\rho_a$ \\
$(8a+2,4b-1)$ & $\eta_0^3\rho_{a,b}$ & $\ol{\eta_0^3\tau^{4b}\per^a}$ & $S_{4b-1}$ & $0$ \\
$(8a+3,4b-1)$ & $\xi_{a,b}$ & $\ol{(2\sqrt{\per})\tau^{4b}\per^a}$ & $E_{2a+1,b}$ & $\xi_a$ \\
 & $\omega_0\eta_0\xi_{a,b}$ & $\ol{\rho\eta_0(2\sqrt{\per})\tau^{4b}\per^a}$ & Above & $0$ \\
\end{longtable}

\begin{longtable}{rrcll}
\caption{Products} \\
\toprule
\endhead
\bottomrule \endfoot
\label{table:products}
\hphantom{\hspace{56pt}}
&$\omega_{a+b}\omega_c$ &$=$ & $ \omega_a\omega_{b+c}$& \ref{lem:root} \\
&$\omega_{a+b}\eta_c$ &$=$ & $ \omega_a\eta_{b+c}$ & \ref{lem:omegaeta} \\
&$\eta_{a+b}\eta_c$ &$=$ & $ \eta_a\eta_{b+c}$ & \ref{lem:root} \\
&$\omega_0^2\eta_a$ &$=$ & $ 2 \omega_a$ & \ref{lem:omegaeta2} \\
&$\eta_0^2 \omega_a$ & $=$ & $ 2 \eta_a$ & \ref{lem:etaomega2} \\
&$\eta_0^3\eta_a$ & $=$ & $ \omega_0^3\omega_{a+1}$ & \ref{lem:root} \\

&$\omega_{a+b}\rho_{c,d}$ & $=$ & $ \omega_a\rho_{b+c,d}$ &  \\
&$\omega_{a+b}\xi_{c,d}$ & $=$ & $ \omega_a\xi_{b+c,d}$ &  \\
&$\eta_{a+b}\rho_{c,d}$ & $=$ & $ \eta_a\rho_{b+c,d}$ &  \\
&$\eta_{a+b}\xi_{c,d}$ & $=$ & $ \eta_a\xi_{b+c,d}$ &  \\
&$\omega_0\xi_{a,b}$ & $=$ & $ \eta_0^3 \rho_{a,b}$ &  \\
&$\eta_0\xi_{a,b}$ & $=$ & $ \omega_0^3 \rho_{a+1,b}$ &  \\

&$\omega_a\cdot\mu_{b,c}$ & $=$ & $\omega_0\mu_{a+b,c}$ & \ref{lem:percomm} \\
&$\omega_a\omega_b \cdot\mu_{c,d}$ & $=$ & $ 2^{j_{2(a+b+c)}-1} \tau^2h\cdot \rho_{a+b+c,d}$ & \ref{lem:omegamu} \\
&$\omega_a\cdot\tau^{2b}h$ & $=$ & $0$ & \ref{lem:rooth1} \\
&$\omega_a \cdot\zeta_{b,c}$ & $=$ & $ \eta_0\mu_0^2 \rho_{a+b,c}$ & \ref{lem:unhidden} \\

&$\eta_a\cdot\mu_{b,c}$ & $=$ & $\eta_0\mu_{a+b,c}$ & \ref{lem:specialper} \\
&$\eta_a \eta_b \cdot\mu_{c,d}$ & $=$ & $ 4 \zeta_{a+b+c,d}$ & \ref{lem:etamu} \\
&$\eta_a \cdot\tau^{4b}h$ & $=$ & $ 0$ & \ref{lem:rooth2} \\
&$\eta_a \cdot\tau^{4b+2}h$ & $=$ & $ \omega_0\mu_0\mu_{a,b}$ & \ref{lem:unhidden} \\
&$\eta_a \cdot\zeta_{b,c}$ & $=$ & $ 0$ &  \\

&$\mu_{a,b}\cdot\mu_{c,d}$ & $=$ & $\mu_0\cdot\mu_{a+c,b+d}$ & \ref{lem:percomm}  \\
&$\mu_0^2\cdot \mu_{a,2b}$ & $=$ & $ 4 \xi_{a,2b+1}$ & \ref{lem:mucubed2} \\
&$\mu_0^2\cdot \mu_{a,2b-1}$ & $=$ & $ (4+2u_{2a+1,2b}\omega_0\eta_0)\xi_{a,2b}$ & \ref{lem:mucubed1} \\
&$\mu_{a,b}\cdot \tau^{4c}h$ & $=$ & $  \omega_0\eta_0\mu_{a,b+c}$ & \ref{lem:muh1} \\
&$\mu_{a,b}\cdot \tau^{4c-2}h$ & $=$ & $ 2^{j_{b+c}-1}\eta_0^2\rho_{a,b+c}$ & \ref{lem:muh2} \\
&$\mu_{a,b}\cdot\rho_{c,d}$ & $=$ & $ \mu_{0} \rho_{a+c,b+d}$ & \\
&$\mu_0^3 \cdot\rho_{a,b}$ & $=$ & $0$ &  \\
&$\mu_{a,b} \cdot \zeta_{c,d}$ & $=$ & $0$ &  \\
&$\mu_{a,b}\cdot\xi_{c,d}$ & $=$ & $ 0$ &  \\

&$\tau^{2a}h\cdot\tau^{2b}h$ & $=$ & $ 2 \tau^{2(a+b)}h$ & \ref{lem:hh} \\
&$\tau^{4a}h\cdot \rho_{b,c}$ & $=$ & $(2-\omega_0\eta_0)\rho_{b,a+c}$ &  \\
&$\tau^{4a+2}h\cdot \rho_{b,c}$ & $=$ & $\tau^{2}h\cdot \rho_{b,a+c}$ & \\
&$\tau^{4a}h\cdot\xi_{b,c}$ & $=$ & $(2-\omega_0\eta_0)\xi_{b,a+c}$ &  \\
&$\tau^{4a+2}h\cdot\xi_{b,c}$ & $=$ & $2\zeta_{b,a+c}$ &  \\
&$\tau^{4a}h\cdot\zeta_{b,c}$ & $=$ & $2\zeta_{b,a+c}$ &  \\
&$\tau^{4a-2}h\cdot \zeta_{b,c}$ & $=$ & $(2-\omega_0\eta_0)\xi_{b,a+c}$&  \\

&$\zeta_{a,b} \cdot \zeta_{c,d}$ & $=$ & $ 0$ &  \\
&$\zeta_{a,b} \cdot \rho_{c,d}$ & $=$ & $ 0$ &  \\
&$\zeta_{a,b} \cdot \xi_{c,d}$ & $=$ & $ 0$ &  \\

&$\rho_{a,b}\cdot\rho_{c,d}$ & $=$ & $ 0$ &  \\
&$\rho_{a,b}\cdot\xi_{c,d}$ & $=$ & $ 0$ &  \\
&$\xi_{a,b}\cdot \xi_{c,d}$ & $=$ & $ 0$ & 
\end{longtable}

\begin{longtable}{lll}
\caption{Multiplicative generators}\label{table:mult} \\
\toprule
$(s,c)$ & Generator & Image in $\pi_\ast S_{K(1)}$ \\
\midrule \endhead
\bottomrule \endfoot
$(8a-1,0)$ & $\omega_a$ & $2^{j_{2a}-1}\rho_a$\\
$(8a+1,0)$ & $\eta_a$ & $\mu_a$  \\
$(0,2b)$, $b\neq 0$ & $\tau^{2b}h$ & $2$ \\
$(8a+1,4b+1)$ & $\mu_{a,b}$  & $\mu_a$ \\
$(8a+3,4b+1)$ & $\zeta_{a,b}$ & $\xi_a$ \\
$(8a-1,4b-1)$ & $\rho_{a,b}$ & $\rho_a$ \\
$(8a+3,4b-1)$ & $\xi_{a,b}$ & $\xi_a$
\end{longtable}

\begin{longtable}{llll}
\caption{Transfers} \\
\toprule
$(s,c)$ & $\alpha\in \pi_s S_{K(1)}$ & $\tr_c(\alpha) \in \pi_{s,c}b(S_{K(1)})$ \\
\midrule \endhead
\bottomrule \endfoot
\label{table:transfers}
$\!\!(0,4b)$ & $1$ & $\tau^{4b}h$ \\
$(0,4b+1)$ & $1$ & $\omega_0\mu_{0,b}$ \\
$(0,4b+2)$ & $1$ & $\tau^{4b+2}h$ \\
$(0,4b-1)$ & $1$ & $2^{j_b-1}\eta_0\rho_{0,b}$ \\
\midrule
$(8a-1,4b)$ & $\rho_a$ & $\omega_0\mu_0\rho_{a,b}$ \\
$(8a,4b)$ & $\mu_0\rho_a$ & $\omega_0\eta_0\mu_0\rho_{a,b}$ \\
$(8a+1,4b)$ & $\mu_a$ & $0$ \\
$(8a+1,4b)$ & $\mu_0^2\rho_a$ & $0$ \\
$(8a+2,4b)$ & $\mu_0\mu_a$ & $0$ \\
$(8a+3,4b)$ & $\xi_a$ & $0$ \\
\midrule
$(8a-1,4b+1)$ & $\rho_a$ & $\tau^2h\cdot \rho_{a,b}$ \\
$(8a,4b+1)$ & $\mu_0\rho_a$ & $\omega_0\mu_0^2\rho_{a,b}$ \\
$(8a+1,4b+1)$ & $\mu_a$ & $\omega_0\eta_0\mu_{a,b}$ \\
$(8a+1,4b+1)$ & $\mu_0^2\rho_a$ & $\omega_0\eta_0\mu_0^2\rho_{a,b}$ \\
$(8a+2,4b+1)$ & $\mu_0\mu_a$ & $0$ \\
$(8a+3,4b+1)$ & $\xi_a$ & $2\zeta_{a,b}$ \\
\midrule
$(8a-1,4b+2)$ & $\rho_a$ & $0$ \\
$(8a,4b+2)$ & $\mu_0\rho_a$ & $0$ \\
$(8a+1,4b+2)$ & $\mu_a$ & $\omega_0\mu_0\mu_{a,b}$ \\
$(8a+1,4b+2)$ & $\mu_0^2\rho_a$ & $0$ \\
$(8a+2,4b+2)$ & $\mu_0\mu_a$ & $\omega_0\eta_0\mu_0\mu_{a,b}$ \\
$(8a+3,4b+2)$ & $\xi_a$ & $0$ \\
\midrule
$(8a-1,4b-1)$ & $\rho_a$ & $(2-\omega_0\eta_0)\rho_{a,b}$ \\
$(8a,4b-1)$ & $\mu_0\rho_a$ & $0$ \\
$(8a+1,4b-1)$ & $\mu_a$ & $2^{j_b-1}\eta_0^2 \rho_{a,b}$ \\
$(8a+1,4b-1)$ & $\mu_0^2\rho_a$ & $0$ \\
$(8a+2,4b-1)$ & $\mu_0\mu_a$ & $2^{j_b-1}\eta_0^3\rho_{a,b}$ \\
$(8a+3,4b-1)$ & $\xi_a$ & $(2-\omega_0\eta_0)\xi_{a,b}$
\end{longtable}

\begingroup
\raggedright
\bibliography{refs}

\begin{thebibliography}{GHIR20}

\bibitem[Ada62]{adams1962vector}
J.~F. Adams.
\newblock Vector fields on spheres.
\newblock {\em Ann. of Math. (2)}, 75:603--632, 1962.

\bibitem[Ada66]{adams1966on}
J.~F. Adams.
\newblock On the groups {$J(X)$}. {IV}.
\newblock {\em Topology}, 5:21--71, 1966.

\bibitem[AI82]{arakiiriye1982equivariant}
Sh\^{o}r\^{o} Araki and Kouyemon Iriye.
\newblock Equivariant stable homotopy groups of spheres with involutions. {I}.
\newblock {\em Osaka Math. J.}, 19(1):1--55, 1982.

\bibitem[BG95]{brunergreenlees1995bredon}
Robert Bruner and John Greenlees.
\newblock The {Bredon}-{L{\"o}ffler} conjecture.
\newblock {\em Exp. Math.}, 4(4):289--297, 1995.

\bibitem[BGI21]{belmontguillouisaksen2021c2}
Eva Belmont, Bertrand~J. Guillou, and Daniel~C. Isaksen.
\newblock {$C_2$}-equivariant and {$\mathbb{R}$}-motivic stable stems {II}.
\newblock {\em Proc. Amer. Math. Soc.}, 149(1):53--61, 2021.

\bibitem[BGL22]{bhattacharyaguillouli2022rmotivic}
Prasit Bhattacharya, Bertrand Guillou, and Ang Li.
\newblock An {{\(R\)}}-motivic {{\(v_1\)}}-self-map of periodicity {{\(1\)}}.
\newblock {\em Homology Homotopy Appl.}, 24(1):299--324, 2022.

\bibitem[BI20a]{belmontisaksen2020rcharts}
Eva {Belmont} and Daniel~C. {Isaksen}.
\newblock {$R$}-motivic {A}dams charts.
\newblock \url{https://s.wayne.edu/isaksen/adams-charts/}, 2020.

\bibitem[BI20b]{belmontisaksen2020rmotivic}
Eva {Belmont} and Daniel~C. {Isaksen}.
\newblock {R-motivic stable stems}.
\newblock {\em arXiv e-prints}, page arXiv:2001.03606, January 2020.

\bibitem[Bou79]{bousfield1979localization}
A.~K. Bousfield.
\newblock The localization of spectra with respect to homology.
\newblock {\em Topology}, 18(4):257--281, 1979.

\bibitem[Bre67]{bredon1967equivariant}
Glen~E. Bredon.
\newblock Equivariant stable stems.
\newblock {\em Bull. Amer. Math. Soc.}, 73:269--273, 1967.

\bibitem[BS20]{behrensshah2020c2}
Mark Behrens and Jay Shah.
\newblock {$C_2$}-equivariant stable homotopy from real motivic stable
  homotopy.
\newblock {\em Ann. K-Theory}, 5(3):411--464, 2020.

\bibitem[Car22]{carrick2022smashing}
Christian Carrick.
\newblock Smashing localizations in equivariant stable homotopy.
\newblock {\em J. Homotopy Relat. Struct.}, 17(3):355--392, 2022.

\bibitem[Cra80]{crabb1980z2}
M.~C. Crabb.
\newblock {\em {$\mathbf{Z}/2$}-homotopy theory}, volume~44 of {\em London
  Mathematical Society Lecture Note Series}.
\newblock Cambridge University Press, Cambridge-New York, 1980.

\bibitem[DI17a]{duggerisaksen2017low}
Daniel Dugger and Daniel~C. Isaksen.
\newblock Low-dimensional {M}ilnor-{W}itt stems over {$\mathbb{R}$}.
\newblock {\em Ann. K-Theory}, 2(2):175--210, 2017.

\bibitem[DI17b]{duggerisaksen2017z2}
Daniel Dugger and Daniel~C. Isaksen.
\newblock {$\mathbb{Z}/2$} -equivariant and {$\mathbb{R}$}-motivic stable
  stems.
\newblock {\em Proc. Amer. Math. Soc.}, 145(8):3617--3627, 2017.

\bibitem[DM87]{davismahowald1987homotopy}
Donald~M. Davis and Mark Mahowald.
\newblock Homotopy groups of some mapping telescopes.
\newblock In {\em Algebraic topology and algebraic {$K$}-theory ({P}rinceton,
  {N}.{J}., 1983)}, volume 113 of {\em Ann. of Math. Stud.}, pages 126--151.
  Princeton Univ. Press, Princeton, NJ, 1987.

\bibitem[GHIR20]{guillouhillisaksenravenel2020cohomology}
Bertrand~J. Guillou, Michael~A. Hill, Daniel~C. Isaksen, and Douglas~Conner
  Ravenel.
\newblock The cohomology of {$C_2$}-equivariant {$\mathcal{A}(1)$} and the
  homotopy of {${\rm ko}_{C_2}$}.
\newblock {\em Tunis. J. Math.}, 2(3):567--632, 2020.

\bibitem[GI20]{guillouisaksen2019bredonlandweber}
Bertrand~J. Guillou and Daniel~C. Isaksen.
\newblock The {B}redon-{L}andweber region in {$C_2$}-equivariant stable
  homotopy groups.
\newblock {\em Doc. Math.}, 25:1865--1880, 2020.

\bibitem[GS96]{greenleessadofsky1996tate}
J.~P.~C. Greenlees and Hal Sadofsky.
\newblock The {T}ate spectrum of {$v_n$}-periodic complex oriented theories.
\newblock {\em Math. Z.}, 222(3):391--405, 1996.

\bibitem[HKR00]{hopkinskuhnravenel2000generalized}
Michael~J. Hopkins, Nicholas~J. Kuhn, and Douglas~C. Ravenel.
\newblock Generalized group characters and complex oriented cohomology
  theories.
\newblock {\em J. Amer. Math. Soc.}, 13:553--594, 2000.

\bibitem[HMS94]{hopkinsmahowaldsadofsky1994constructions}
Michael~J. Hopkins, Mark Mahowald, and Hal Sadofsky.
\newblock Constructions of elements in {P}icard groups.
\newblock In {\em Topology and representation theory ({E}vanston, {IL}, 1992)},
  volume 158 of {\em Contemp. Math.}, pages 89--126. Amer. Math. Soc.,
  Providence, RI, 1994.

\bibitem[HS96]{hoveysadofsky1996tate}
Mark Hovey and Hal Sadofsky.
\newblock Tate cohomology lowers chromatic {B}ousfield classes.
\newblock {\em Proc. Amer. Math. Soc.}, 124(11):3579--3585, 1996.

\bibitem[Iri82]{iriye1982equivariant}
Kouyemon Iriye.
\newblock Equivariant stable homotopy groups of spheres with involutions. {II}.
\newblock {\em Osaka J. Math.}, 19(4):733--743, 1982.

\bibitem[Jam58]{james1958cross}
I.~M. James.
\newblock Cross-sections of {S}tiefel manifolds.
\newblock {\em Proc. London Math. Soc. (3)}, 8:536--547, 1958.

\bibitem[Lan69]{landweber1969equivariant}
Peter~S. Landweber.
\newblock On equivariant maps between spheres with involutions.
\newblock {\em Ann. of Math. (2)}, 89:125--137, 1969.

\bibitem[Lew88]{lewis1988rog}
L.~Gaunce~jun. Lewis.
\newblock The {RO}({G})-graded equivariant ordinary cohomology of complex
  projective spaces with linear {{\({\mathbb{Z}}/p\)}} actions.
\newblock Algebraic topology and transformation groups, {Proc}. {Conf}.,
  {G{\"o}ttingen}/{FRG} 1987, {Lect}. {Notes} {Math}. 1361, 53-122 (1988).,
  1988.

\bibitem[Lin80]{lin1980conjectures}
Wen~Hsiung Lin.
\newblock On conjectures of {M}ahowald, {S}egal and {S}ullivan.
\newblock {\em Math. Proc. Cambridge Philos. Soc.}, 87(3):449--458, 1980.

\bibitem[Lö78]{loffler1977equivariant}
Peter Löffler.
\newblock Equivariant framability of involutions on homotopy spheres.
\newblock {\em Manuscripta Math.}, 23(2):161--171, 1977/78.

\bibitem[Mah65]{mahowald1965short}
Mark Mahowald.
\newblock A short proof of the {J}ames periodicity of {$\pi
  _{k+p}(V_{k+m,m})$}.
\newblock {\em Proc. Amer. Math. Soc.}, 16:512, 1965.

\bibitem[Mah67]{mahowald1967metastable}
Mark Mahowald.
\newblock {\em The metastable homotopy of { $S^n$ }}.
\newblock Memoirs of the American Mathematical Society, No. 72. American
  Mathematical Society, Providence, R.I., 1967.

\bibitem[Mah82]{mahowald1982image}
Mark Mahowald.
\newblock The image of {$J$} in the {$EHP$} sequence.
\newblock {\em Ann. of Math. (2)}, 116(1):65--112, 1982.

\bibitem[Min83]{minami1983equivariant}
Haruo Minami.
\newblock On equivariant {$J$}-homomorphism for involutions.
\newblock {\em Osaka Math. J.}, 20(1):109--122, 1983.

\bibitem[MR93]{mahowaldravenel1993root}
Mark~E. Mahowald and Douglas~C. Ravenel.
\newblock The root invariant in homotopy theory.
\newblock {\em Topology}, 32(4):865--898, 1993.

\bibitem[Rav84]{ravenel1984localization}
Douglas~C. Ravenel.
\newblock Localization with respect to certain periodic homology theories.
\newblock {\em Amer. J. Math.}, 106(2):351--414, 1984.

\bibitem[Str99]{strickland1999formal}
Neil~P. Strickland.
\newblock Formal schemes and formal groups.
\newblock In {\em Homotopy invariant algebraic structures ({B}altimore, {MD},
  1998)}, volume 239 of {\em Contemp. Math.}, pages 263--352. Amer. Math. Soc.,
  Providence, RI, 1999.

\end{thebibliography}
\bibliographystyle{alpha}
\endgroup

\end{document}